\documentclass[oneside,english]{amsart}
\usepackage[T1]{fontenc}
\usepackage[latin9]{inputenc}
\usepackage{prettyref}
\usepackage{amstext}
\usepackage{amsthm}
\usepackage{amssymb}
\usepackage[numbers]{natbib}
\usepackage[all]{xy}

\makeatletter
\numberwithin{equation}{section}
\numberwithin{figure}{section}
\theoremstyle{plain}
\newtheorem{thm}{\protect\theoremname}[section]
\theoremstyle{remark}
\newtheorem{rem}[thm]{\protect\remarkname}
\theoremstyle{definition}
\newtheorem{defn}[thm]{\protect\definitionname}
\theoremstyle{plain}
\newtheorem{cor}[thm]{\protect\corollaryname}
\theoremstyle{plain}
\newtheorem{prop}[thm]{\protect\propositionname}
\theoremstyle{plain}
\newtheorem{fact}[thm]{\protect\factname}
\theoremstyle{definition}
\newtheorem{example}[thm]{\protect\examplename}
\theoremstyle{plain}
\newtheorem{lem}[thm]{\protect\lemmaname}
\theoremstyle{remark}
\newtheorem{notation}[thm]{\protect\notationname}

\usepackage{mathtools}
\usepackage{braket}
\usepackage{prettyref}
\usepackage{graphicx}
\usepackage{tikz}
\usetikzlibrary{patterns}
\usetikzlibrary{calc}

\newrefformat{prop}{Proposition \ref{#1}}
\newrefformat{cor}{Corollary \ref{#1}}
\newrefformat{exa}{Example \ref{#1}}
\newrefformat{fact}{Fact \ref{#1}}
\newrefformat{subsec}{Subsection \ref{#1}}

\mathchardef\mhyphen="2D

\global\long\def\ns#1{\prescript{\ast}{}{#1}}
\global\long\def\ss#1{\prescript{S}{}{#1}}
\global\long\def\ssig#1{\prescript{\sigma}{}{#1}}
\global\long\def\st#1{\prescript{\circ}{}{#1}}
\newcommand{\FIN}{\operatorname{FIN}}
\newcommand{\INF}{\operatorname{INF}}
\newcommand{\NS}{\operatorname{NS}}
\newcommand{\CPT}{\operatorname{CPT}}
\newcommand{\PNS}{\operatorname{PNS}}
\newcommand{\PCPT}{\operatorname{PCPT}}
\newcommand{\NF}{\operatorname{NF}}
\newcommand{\cl}{\operatorname{cl}}
\newcommand{\intr}{\operatorname{int}}
\newcommand{\img}{\operatorname{im}}
\newcommand{\id}{\operatorname{id}}
\newcommand{\hexp}{\mhyphen\exp}
\newcommand{\Ult}{\operatorname{Ult}}
\newcommand{\bUlt}{\operatorname{\flat Ult}}

\@ifundefined{showcaptionsetup}{}{%
 \PassOptionsToPackage{caption=false}{subfig}}
\usepackage{subfig}
\makeatother

\usepackage{babel}
\providecommand{\corollaryname}{Corollary}
\providecommand{\definitionname}{Definition}
\providecommand{\examplename}{Example}
\providecommand{\factname}{Fact}
\providecommand{\lemmaname}{Lemma}
\providecommand{\notationname}{Notation}
\providecommand{\propositionname}{Proposition}
\providecommand{\remarkname}{Remark}
\providecommand{\theoremname}{Theorem}

\begin{document}
\title{Nonstandard methods in large-scale topology II}
\author{Takuma Imamura}
\address{Research Institute for Mathematical Sciences\\
Kyoto University\\
Kitashirakawa Oiwake-cho, Sakyo-ku, Kyoto 606-8502, Japan}
\curraddr{AaaS Bridge, inc.\\
2-1-3 Amanuma, Suginami-ku, Tokyo 167-0032, Japan}
\email{imamura.takuma.66s@kyoto-u.jp}
\begin{abstract}
This paper is a sequel of \citep{Ima19} where we set up a framework
of nonstandard large-scale topology. In the present paper, we apply
our framework to various topics in large-scale topology: spaces having
with both small-scale and large-scale structures, large-scale structures
on nonstandard extensions, size properties of subsets of coarse spaces,
and coarse hyperspaces.
\end{abstract}

\maketitle

\section*{Introduction}

This paper is a continuation of the paper \citep{Ima19}. In the preceding
paper, we set up a framework for treating (pre)bornological spaces
and coarse spaces in nonstandard analysis, where a prebornological
space is a generalisation of a bornological space which better fits
(non-connected) coarse spaces. In the present paper, we apply our
framework to various topics in large-scale topology: spaces having
both small-scale and large-scale structures, large-scale structures
on nonstandard extensions, size properties of subsets of coarse spaces,
and coarse hyperspaces. The present paper is organised as follows.

Given a standard space $X$, its nonstandard extension $\ns{X}$ contains
several types of nonstandard points, such as nearstandard points $\NS\left(X\right)$
(for topological spaces), prenearstandard points $\PNS\left(X\right)$
(for uniform spaces), and finite points $\FIN\left(X\right)$ (for
bornological spaces). It is well-known that an inclusion of one class
into another characterises various (standard) properties of $X$.
For instance, \citet{Rob66} proved the following celebrated theorem:
$X$ is compact if and only if $\ns{X}\subseteq\NS\left(X\right)$.
In \prettyref{sec:Nonstandard-points}, we identify seven classes
of nonstandard points (including $X$ and $\ns{X}$), and complete
the correspondence between the properties of $X$ and the inclusion
relations among seven classes (see \prettyref{fig:Structure-of-nonstandard-points}).
The properties include von Neumann completeness, properness, and various
compatibility conditions between small-scale and large-scale structures.
In particular, a new compatibility condition, called weak u-II-compatibility,
is extracted from its nonstandard characterisation.

In \prettyref{sec:Large-scale-structures-on-ns-ext}, we explore large-scale
structures of nonstandard spaces. \citet{KK16} introduced bornologies
on nonstandard extensions $\ns{X}$ of bornological spaces $X$, called
S-bornologies. S-bornology is a large-scale counterpart of S-topology
\citep{Lux69,ST00,Str72}. Generalising to prebornological spaces
and coarse spaces, we obtain the notions of S-prebornology and S-coarse
structure. Firstly, we deal with S-prebornological structures. It
is well-known that the Stone\textendash \v{C}ech compactification
$\beta X$ can be obtained as a quotient of the S-topological space
$S^{t}X$ \citep{Lux69,Str72}. It is natural to consider its large-scale
analogue. To do this, for each prebornological space $X$, we define
a new prebornological space $\bUlt X$, which consists of appropriate
ultrafilters on $X$. We then show that $\bUlt X$ can be represented
as a quotient of the S-prebornological space $SX$. Secondly, we deal
with S-coarse structures. The S-boundary $\partial_{S}X$ of a coarse
space $X$ is defined as a subspace of the S-coarse space $S^{c}X$
consisting of all infinite points. We prove that the coarse structure
of $X$ can be recovered from the induced prebornology of $X$ and
the coarse structure of $\partial_{S}X$.

In \prettyref{sec:Large-scale-structures-on-Pow}, we are devoted
to studying various size properties of subsets of coarse spaces, which
originally arose in the context of group theory \citep{BM99,BM01},
and were extended to general coarse spaces \citep{PB03,PZ07}. In
the first half of this section, we provide some nonstandard characterisations
of the size properties. We then give nonstandard proofs for some (known)
fundamental results, such as lattice-theoretic criteria for extralargeness
and smallness. The relationship among thin coarse spaces, satellite
coarse spaces and slowly oscillating maps is also discussed. Interestingly,
despite the \emph{large-scale} nature of these results, some of our
proofs will be evident from the elementary knowledge of \emph{small-scale}
topology. In the last half of this section, we concern with natural
coarse structures on powersets of coarse spaces. Given a metric space
$X$, its powerset $\mathcal{P}\left(X\right)$ has the Hausdorff
metric $d_{H}$, and thereby can be considered as a uniform and coarse
space. This construction can be generalised to arbitrary uniform and
coarse spaces, and leads to the notions of uniform hyperspaces \citep{Bou07}
and coarse hyperspaces \citep{PP18,DPPZ19}. We prove some theorems
on coarse hyperspaces, including the characterisation of thinness
in terms of hyperspaces.

\section{Preliminaries}

We refer to \citep{HN77} for bornology, \citep{Roe03} for coarse
topology (in terms of coarse spaces), \citep{PB03,PZ07} for coarse
topology in terms of balleans, \citep{Dav05,DD95,Rob66,SL76} for
nonstandard (small-scale) topology. We also refer to the surveys \citep{Pro11,PP18b,Zav19}
for size properties, coarse hyperspaces and their use in group theory.

\subsection{Notation and terminology}
\begin{enumerate}
\item Let $X$ be a set, $E\subseteq X\times X$ and $A\subseteq X$. The
\emph{$E$-closure} of $A$ is the set defined by $E\left[A\right]=\bigcup_{x\in A}E\left[x\right]$,
where $E\left[x\right]=\set{y\in X|\left(x,y\right)\in E}$. The\emph{
$E$-interior} of $A$ is the set defined by $\intr_{X,E}A=\set{x\in X|E\left[x\right]\subseteq A}$.
The $E$-closure and the $E$-interior are related with each other
as follows: $E\left[A\right]=X\setminus\intr_{X,E^{-1}}\left(X\setminus A\right)$;
$\intr_{X,E}A=X\setminus E^{-1}\left[X\setminus A\right]$.
\item Let $\left(X,\mathcal{T}_{X}\right)$ be a standard topological space.
The \emph{monad} of a point $x\in X$ is the set $\mu_{X}\left(x\right)=\bigcap\set{\ns{U}|x\in U\in\mathcal{T}_{X}}$.
The elements of $\NS\left(X\right)=\bigcup_{x\in X}\mu_{X}\left(x\right)$
are called \emph{nearstandard points}. The non-nearstandard points
of $\ns{X}$ are called \emph{remote points}.
\item Let $\left(X,\mathcal{U}_{X}\right)$ be a standard uniform space.
We say that two points $x,y\in\ns{X}$ are \emph{infinitely close}
(write $x\approx_{X}y$) if $\left(x,y\right)\in\ns{U}$ holds for
all $U\in\mathcal{U}_{X}$. The \emph{(uniform) monad} of a point
$x\in\ns{X}$ is the set $\mu_{X}^{u}\left(x\right)=\bigcap_{U\in\mathcal{U}_{X}}\ns{U}\left[x\right]=\set{y\in\ns{X}|x\approx_{X}y}$.
For each (standard) $x\in X$, $\mu_{X}^{u}\left(x\right)=\mu_{X}\left(x\right)$
holds. The elements of $\PNS\left(X\right)=\bigcap_{U\in\mathcal{U}_{X}}\bigcup_{x\in X}\ns{U}\left[x\right]$
are called \emph{prenearstandard points}.
\item A \emph{prebornology} on a set $X$ is a family $\mathcal{B}_{X}$
of subsets of $X$ satisfying the following conditions: (i) $\bigcup\mathcal{B}_{X}=X$;
(ii) if $A\subseteq B\in\mathcal{B}_{X}$, then $A\in\mathcal{B}_{X}$;
(iii) if $A,B\in\mathcal{B}_{X}$ and $A\cap B\neq\varnothing$, then
$A\cup B\in\mathcal{B}_{X}$. The prebornology $\mathcal{B}_{X}$
is called a \emph{bornology} if $A\cup B\in\mathcal{B}_{X}$ holds
for arbitrary $A,B\in\mathcal{B}_{X}$. The point of this generalisation
is that every coarse structure induces a prebornology, that is not
necessarily a bornology. Now, let $\left(X,\mathcal{B}_{X}\right)$
be a standard prebornological space. The \emph{galaxy} of a point
$x\in X$ is the set $G_{X}\left(x\right)=\bigcup\set{\ns{B}|x\in B\in\mathcal{B}_{X}}$.
The elements of $\FIN\left(X\right)=\bigcup_{x\in X}G_{X}\left(x\right)$
are called \emph{finite points}. The non-finite points of $\ns{X}$
are called \emph{infinite points}. The set of all infinite points
of $\ns{X}$ is denoted by $\INF\left(X\right)$, i.e., $\INF\left(X\right)=\ns{X}\setminus\FIN\left(X\right)$.
\item Let $\left(X,\mathcal{C}_{X}\right)$ be a standard coarse space.
We say that two points $x,y\in\ns{X}$ are \emph{finitely close} (write
$x\sim_{X}y$) if $\left(x,y\right)\in\ns{E}$ holds for some $E\in\mathcal{C}_{X}$.
The \emph{(coarse) galaxy} of a point $x\in\ns{X}$ is the set $G_{X}^{c}\left(x\right)=\bigcup_{E\in\mathcal{C}_{X}}\ns{E}\left[x\right]=\set{y\in\ns{X}|x\sim_{X}y}$.
The coarse structure $\mathcal{C}_{X}$ induces a prebornology $\mathcal{B}_{X}=\set{B\subseteq X|B\times B\in\mathcal{C}_{X}}$.
For each (standard) $x\in X$, $G_{X}^{c}\left(x\right)=G_{X}\left(x\right)$
holds by \citep[Proposition 3.12]{Ima19}.
\item A map $f\colon X\to Y$ between prebornological spaces is said to
be \emph{bornological} if $f\left(B\right)\in\mathcal{B}_{Y}$ for
all $B\in\mathcal{B}_{X}$; and $f$ is \emph{proper} if $f^{-1}\left(B\right)\in\mathcal{B}_{X}$
holds for all $B\in\mathcal{B}_{Y}$. On the other hand, a map $f\colon X\to Y$
between coarse spaces is said to be \emph{bornologous} (or \emph{uniformly
bornological}) if $\left(f\times f\right)\left(E\right)=\set{\left(f\left(x\right),f\left(y\right)\right)|\left(x,y\right)\in E}\in\mathcal{C}_{Y}$
holds for all $E\in\mathcal{C}_{X}$.
\end{enumerate}
\begin{rem}
A prebornological space is \emph{not} a large-scale counterpart of
a pretopological space, also known as a \v{C}ech closure space (see
\citep[Definition 14 A.1]{CFK66} for the definition). Comparing the
local definition of prebornology \citep[Lemma 2.3]{Ima19} with that
of pretopology \citep[Theorem 14 B.10]{CFK66}, the assumption of
emptiness, (BN1) and (BN3) correspond to (nbd1), (nbd2) and (nbd3),
respectively; (BN2) corresponds to the last statement of \citep[Theorem 14 B.3]{CFK66};
however, there is no counterpart of (BN4). In fact, a prebornological
space is a large-scale counterpart of a topological space: (BN4) corresponds
to (nbd4) of the local definition of topology \citep[Theorem 15 A.4]{CFK66}.
Because of this, it might be better to rename `prebornology' to something
more appropriate. Some researchers use the term `bounded structure'
instead (e.g. \citep{Dyd19}).
\end{rem}

\subsection{Coarse galaxy and galactic core}

Galaxy is a key concept in nonstandard large-scale topology, as we
have demonstrated in \citep{Ima19}. It can be considered as the nonstandard
counterpart of coarse closure. We here introduce the notion of galactic
core as the nonstandard counterpart of coarse interior.
\begin{defn}
Let $X$ be a standard coarse space and $A\subseteq\ns{X}$. The \emph{(coarse)
galaxy} of $A$ is the set defined by
\begin{align*}
G_{X}^{c}\left(A\right) & =\set{x\in\ns{X}|x\sim_{X}a\text{ for some }a\in A}\\
 & =\bigcup_{a\in A}G_{X}^{c}\left(a\right)\\
 & =\bigcup_{E\in\mathcal{C}_{X}}\ns{E}\left[A\right].
\end{align*}
The \emph{(coarse) galactic core }of $A$ is the set defined by
\begin{align*}
C_{X}^{c}\left(A\right) & =\set{x\in\ns{X}|x\sim_{X}y\text{ for no }y\in\ns{X}\setminus A}\\
 & =\set{x\in\ns{X}|G_{X}^{c}\left(x\right)\subseteq A}\\
 & =\bigcap_{E\in\mathcal{C}_{X}}\mathop{\ns{\left(\intr_{X,E}\right)}}A.
\end{align*}
\end{defn}

The galaxy map and the galactic core map behave like topological closure
and interior.
\begin{thm}
\label{thm:galaxy-and-core}Let $X$ be a standard coarse space, $A$,
$B$ and $A_{i}\ \left(i\in I\right)$ subsets of $\ns{X}$.
\begin{enumerate}
\item \label{enu:G-is-closure}$G_{X}^{c}$ is a closure operator on $\mathcal{P}\left(\ns{X}\right)$:
\begin{enumerate}
\item \label{enu:A-sub-GA}$A\subseteq G_{X}^{c}\left(A\right)$;
\item \label{enu:G-empty=00003Dempty}$G_{X}^{c}\left(\varnothing\right)=\varnothing$;
\item \label{enu:GG=00003DG}$G_{X}^{c}\left(G_{X}^{c}\left(A\right)\right)=G_{X}^{c}\left(A\right)$;
\item \label{enu:G-union=00003Dunion-G}$G_{X}^{c}\left(\bigcup_{i\in I}A_{i}\right)=\bigcup_{i\in I}G_{X}^{c}\left(A_{i}\right)$.
\end{enumerate}
\item $C_{X}^{c}$ is an interior operator on $\mathcal{P}\left(\ns{X}\right)$:
\begin{enumerate}
\item $C_{X}^{c}\left(A\right)\subseteq A$;
\item $C_{X}^{c}\left(\ns{X}\right)=\ns{X}$;
\item $C_{X}^{c}\left(C_{X}^{c}\left(A\right)\right)=C_{X}^{c}\left(A\right)$;
\item $C_{X}^{c}\left(\bigcap_{i\in I}A_{i}\right)=\bigcap_{i\in I}C_{X}^{c}\left(A_{i}\right)$.
\end{enumerate}
\item \label{enu:C-is-neg-G-neg}$\ns{X}\setminus C_{X}^{c}\left(A\right)=G_{X}^{c}\left(\ns{X}\setminus A\right)$.
\item \label{enu:GC-eq-C}$G_{X}^{c}\left(C_{X}^{c}\left(A\right)\right)=C_{X}^{c}\left(A\right)$.
\end{enumerate}
\end{thm}

\begin{proof}
\eqref{enu:G-empty=00003Dempty} and \eqref{enu:G-union=00003Dunion-G}
are trivial. \eqref{enu:A-sub-GA} and \eqref{enu:GG=00003DG} immediately
follow from the reflexivity and the transitivity of $\sim_{X}$, respectively.
\eqref{enu:C-is-neg-G-neg} and \eqref{enu:GC-eq-C} follow from the
symmetricity and the transitivity of $\sim_{X}$, respectively. Let
us only prove \eqref{enu:GC-eq-C}: let $x\in G_{X}^{c}\left(C_{X}^{c}\left(A\right)\right)$.
There exists a $y\in C_{X}^{c}\left(A\right)$ such that $x\sim_{X}y$.
By the transitivity of $\sim_{X}$, we have that $G_{X}^{c}\left(x\right)\subseteq G_{X}^{c}\left(y\right)\subseteq\ns{A}$,
and therefore $x\in C_{X}^{c}\left(A\right)$. The reverse inclusion
follows from \eqref{enu:A-sub-GA}.
\end{proof}
\begin{cor}
\label{cor:galactic-topology-of-starX}Let $X$ be a standard coarse
space. There exists a (unique) topology on $\ns{X}$ such that the
closure and the interior operators are given by $G_{X}^{c}$ and $C_{X}^{c}$.
This topology is almost discrete (in the sense that every open set
is closed, and vice versa).
\end{cor}

\begin{rem}
The finite part $\FIN\left(X\right)$ is equal to $G_{X}^{c}\left(X\right)$
by definition. $X$ can be considered as a $G_{X}^{c}$-dense subset
of $\FIN\left(X\right)$; and $\FIN\left(X\right)$ can be considered
as a $G_{X}^{c}$-closed subset of $\ns{X}$.
\end{rem}

\subsection{Asymorphisms and coarse equivalences}

We provide nonstandard characterisations of asymorphisms and coarse
equivalences which will be used throughout.
\begin{defn}[Standard]
Let $X$ and $Y$ be coarse spaces. A map $f\colon X\to Y$ is said
to be
\begin{enumerate}
\item \emph{effectively proper} if $\left(f^{-1}\times f^{-1}\right)\left(E\right)=\set{\left(x,y\right)\in X\times X|\left(f\left(x\right),f\left(y\right)\right)\in E}\in\mathcal{C}_{X}$
holds for any $E\in\mathcal{C}_{Y}$;
\item an \emph{asymorphism} if it is a bornologous bijection with a bornologous
inverse;
\item an \emph{asymorphic embedding} if $f$ is an asymorphism between $X$
and $\img\left(f\right)$.
\end{enumerate}
\end{defn}

\begin{thm}
\label{thm:nonst-charact-eff-proper}Let $X$ and $Y$ be standard
coarse spaces and let $f\colon X\to Y$ be a map. The following are
equivalent:
\begin{enumerate}
\item \label{enu:nonst-charact-eff-proper-cond-1}$f$ is effectively proper;
\item \label{enu:nonst-charact-eff-proper-cond-2}for any $x,y\in\ns{X}$,
$\ns{f}\left(x\right)\sim_{Y}\ns{f}\left(y\right)$ implies $x\sim_{X}y$.
\end{enumerate}
\end{thm}

\begin{proof}
\eqref{enu:nonst-charact-eff-proper-cond-1}$\Rightarrow$\eqref{enu:nonst-charact-eff-proper-cond-2}:
Let $E\in\mathcal{C}_{Y}$ with $\left(\ns{f}\left(x\right),\ns{f}\left(y\right)\right)\in\ns{E}$.
Then $\left(x,y\right)\in\ns{\left(\left(f^{-1}\times f^{-1}\right)\left(E\right)\right)}$.
Since $f$ is effectively proper, $\left(f^{-1}\times f^{-1}\right)\left(E\right)\in\mathcal{C}_{X}$.
Hence $x\sim_{X}y$.

\eqref{enu:nonst-charact-eff-proper-cond-2}$\Rightarrow$\eqref{enu:nonst-charact-eff-proper-cond-1}:
Let $E\in\mathcal{C}_{Y}$. For any $x,y\in\ns{X}$ with $\left(\ns{f}\left(x\right),\ns{f}\left(y\right)\right)\in\ns{E}$,
we have that $\ns{f}\left(x\right)\sim_{Y}\ns{f}\left(y\right)$,
so $x\sim_{X}y$ by assumption. On the other hand, there exists an
$F\in\ns{\mathcal{C}_{X}}$ such that ${\sim_{X}}\subseteq F$ by
\prettyref{lem:General-Approximation-Lemma} (to be proved in \ref{sec:Overspill-and-underspill}).
Hence $\ns{\left(\left(f^{-1}\times f^{-1}\right)\left(E\right)\right)}\subseteq{\sim_{X}}\subseteq F\in\ns{\mathcal{C}_{X}}$.
Therefore $\left(f^{-1}\times f^{-1}\right)\left(E\right)\in\mathcal{C}_{X}$
by transfer.
\end{proof}
\begin{prop}[Standard]
\label{prop:effectively-properness-and-inverse}Let $X$ and $Y$
be coarse spaces and let $f\colon X\to Y$ be a bijection. The following
are equivalent:
\begin{enumerate}
\item $f$ has a bornologous inverse;
\item $f$ is effectively proper.
\end{enumerate}
\end{prop}

\begin{proof}
$f^{-1}$ is bornologous $\iff$ $x\sim_{Y}y$ implies $\ns{f^{-1}}\left(x\right)\sim_{X}\ns{f^{-1}}\left(y\right)$
for all $x,y\in\ns{Y}$ (by \citep[Theorem 3.23]{Ima19}) $\iff$
$\ns{f}\left(x\right)\sim_{Y}\ns{f}\left(y\right)$ implies $x\sim_{X}y$
for all $x,y\in\ns{X}$ (by bijectivity) $\iff$ $f$ is effectively
proper (by \prettyref{thm:nonst-charact-eff-proper}).
\end{proof}
\begin{cor}
\label{cor:asymorphism-preserves-G-and-C}Let $f\colon X\to Y$ be
a standard asymorphism. Then $\ns{f}\circ G_{X}^{c}=G_{Y}^{c}\circ\ns{f}$
and $\ns{f}\circ C_{X}^{c}=C_{Y}^{c}\circ\ns{f}$. In other words,
$\ns{f}\colon\ns{X}\to\ns{Y}$ is a homeomorphism with respect to
the topology defined in \prettyref{cor:galactic-topology-of-starX}.
\end{cor}

\begin{proof}
Immediate from the nonstandard characterisation of asymorphisms (\citep[Theorem 3.23]{Ima19}
and \prettyref{thm:nonst-charact-eff-proper}).
\end{proof}
\begin{defn}[Standard]
Let $X$ and $Y$ be coarse spaces. A map $f\colon X\to Y$ is said
to be
\begin{enumerate}
\item \emph{coarsely surjective} if there exists an $E\in\mathcal{C}_{Y}$
such that $E\left[f\left(X\right)\right]=Y$;
\item a \emph{coarse equivalence} (a.k.a.\emph{ bornotopy equivalence})
if it is a bornologous map with a bornotopy inverse (a bornologous
map $g\colon Y\to X$ such that $g\circ f$ and $f\circ g$ are bornotopic
to $\id_{X}$ and $\id_{Y}$, respectively).
\end{enumerate}
\end{defn}

\begin{thm}
\label{thm:nonst-charact-coarse-surjection}Let $X$ and $Y$ be standard
coarse spaces and let $f\colon X\to Y$ be a map. The following are
equivalent:
\begin{enumerate}
\item $f$ is coarsely surjective;
\item $G_{Y}^{c}\left(\ns{f}\left(\ns{X}\right)\right)=\ns{Y}$.
\end{enumerate}
\end{thm}

\begin{proof}
Suppose $f$ is coarsely surjective, i.e., there exists an $E\in\mathcal{C}_{Y}$
such that $E\left[f\left(X\right)\right]=Y$. By transfer, $\ns{Y}=\ns{E}\left[\ns{f}\left(\ns{X}\right)\right]\subseteq G_{Y}^{c}\left(\ns{f}\left(\ns{X}\right)\right)\subseteq\ns{Y}$.
Hence $G_{Y}^{c}\left(\ns{f}\left(\ns{X}\right)\right)=\ns{Y}$.

Conversely, suppose $G_{Y}^{c}\left(\ns{f}\left(\ns{X}\right)\right)=\ns{Y}$.
By \prettyref{lem:General-Approximation-Lemma}, there exists an $E\in\ns{\mathcal{C}_{Y}}$
such that ${\sim_{Y}}\subseteq E$. Then $\ns{Y}=G_{Y}^{c}\left(\ns{f}\left(\ns{X}\right)\right)\subseteq E\left[\ns{f}\left(\ns{X}\right)\right]\subseteq\ns{Y}$,
so $E\left[\ns{f}\left(\ns{X}\right)\right]=\ns{Y}$. By transfer,
there exists an $F\in\mathcal{C}_{Y}$ such that $F\left[f\left(X\right)\right]=Y$.
\end{proof}
\begin{prop}[Standard]
\label{prop:coarse-surjectivity-and-bornotopy-inverse}Let $X$ and
$Y$ be coarse spaces and let $f\colon X\to Y$ be a map. The following
are equivalent:
\begin{enumerate}
\item $f$ has a bornotopy inverse;
\item $f$ is effectively proper and coarsely surjective.
\end{enumerate}
\end{prop}

\begin{proof}
Suppose $f$ has a bornotopy inverse $g\colon Y\to X$. For any $y\in\ns{Y}$,
$y\sim_{Y}\ns{f}\circ\ns{g}\left(y\right)\in\ns{f}\left(\ns{X}\right)$,
so $y\in G_{X}^{c}\left(\ns{f}\left(\ns{X}\right)\right)$. By \prettyref{thm:nonst-charact-coarse-surjection},
$f$ is coarsely surjective. For any $x,y\in\ns{X}$, if $\ns{f}\left(x\right)\sim_{Y}\ns{f}\left(y\right)$,
then $x\sim_{X}\ns{g}\circ\ns{f}\left(x\right)\sim_{Y}\ns{g}\circ\ns{f}\left(y\right)\sim_{Y}y$
by \citep[Theorem 3.23]{Ima19}. Hence $f$ is effectively proper
by \prettyref{thm:nonst-charact-eff-proper}.

Conversely, suppose $f$ is effectively proper and coarsely surjective.
Let $E\in\mathcal{C}_{X}$ be such that $E\left[f\left(X\right)\right]=Y$.
Choose a (non-unique) map $g\colon Y\to X$ such that $y\in E\left[f\circ g\left(y\right)\right]$
for all $y\in Y$. Clearly $f\circ g$ is bornotopic to $\id_{Y}$.
For any $x\in\ns{X}$, since $\ns{f}\circ\ns{g}\circ\ns{f}\left(x\right)\sim_{Y}\ns{f}\left(x\right)$,
we have that $\ns{g}\circ\ns{f}\left(x\right)\sim_{X}x$ by \prettyref{thm:nonst-charact-eff-proper}.
Hence $g\circ f$ is bornotopic to $\id_{X}$. For any $y,y'\in\ns{Y}$
with $y\sim_{Y}y'$, since $\ns{f}\circ\ns{g}\left(y\right)\sim_{Y}y\sim_{Y}y'\sim_{Y}\ns{f}\circ\ns{g}\left(y\right)$,
it follows that $\ns{g}\left(y\right)\sim_{X}\ns{g}\left(y'\right)$
by \prettyref{thm:nonst-charact-eff-proper}. Hence $g$ is bornologous
by \citep[Theorem 3.23]{Ima19}.
\end{proof}

\section{\label{sec:Nonstandard-points}Several classes of nonstandard points}

Given a standard space $X$ with small-scale and/or large-scale structures,
we have the following classes of nonstandard points:
\begin{align*}
\CPT\left(X\right) & =\bigcup_{K\colon\text{compact}}\ns{K},\\
\NS\left(X\right) & =\bigcup_{x\in X}\mu_{X}\left(x\right),\\
\PCPT\left(X\right) & =\bigcup_{P\colon\text{precompact}}\ns{P},\\
\PNS\left(X\right) & =\bigcap_{U\colon\text{entourage}}\bigcup_{x\in X}\ns{U}\left[x\right],\\
\FIN\left(X\right) & =\bigcup_{B\colon\text{bounded}}\ns{B}.
\end{align*}
where $\CPT\left(X\right)$ and $\NS\left(X\right)$ are defined for
standard topological spaces, $\PCPT\left(X\right)$ and $\PNS\left(X\right)$
for standard uniform spaces, and $\FIN\left(X\right)$ for standard
(pre)bornological spaces. The first four classes have been extensively
studied in the existing literature, while the last class has been
studied only for special cases (such as topological vector spaces
\citep{HM72}). The aim of this section is to clarify the relationship
among those classes (together with $X$ and $\ns{X}$). As we shall
see, the inclusion relations characterise various properties of spaces
(see \prettyref{fig:Structure-of-nonstandard-points}).
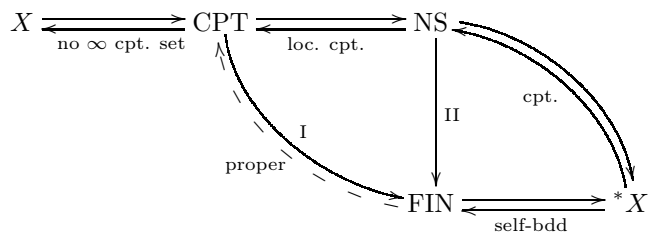
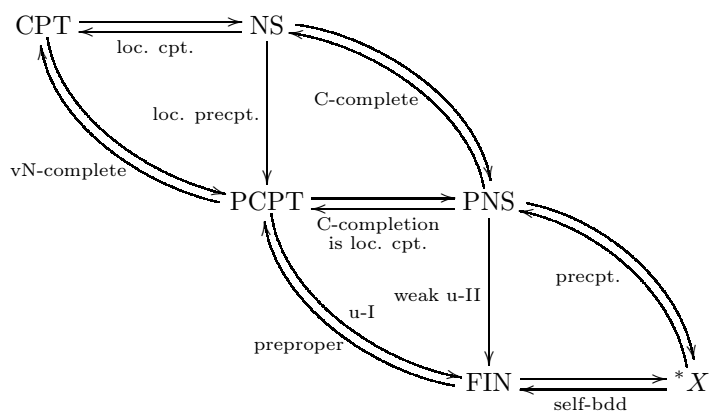
\begin{figure}
\centering\subfloat[topological-bornological]{$\xymatrix{X\ar@<2pt>[rr] &  & \CPT\ar@<2pt>[ll]^{\text{no \ensuremath{\infty} cpt. set}}\ar@<2pt>[rr]\ar@<2pt>@(d,l)[ddrr]^{\text{I}} &  & \NS\ar@<2pt>[ll]^{\text{loc. cpt.}}\ar@<2pt>[dd]^{\text{II}}\ar@<2pt>@(r,u)[ddrr]\\
\\
 &  &  &  & \FIN\ar@<2pt>[rr]\ar@<2pt>@(l,d)@{-->}[uull]^{\text{proper}} &  & \ns{X}\ar@<2pt>[ll]^{\text{self-bdd}}\ar@<2pt>@(u,r)[uull]^{\text{cpt.}}
}
$

}

\subfloat[uniform-bornological]{$\xymatrix{\CPT\ar@<2pt>@(d,l)[ddrr]\ar@<2pt>[rr] &  & \NS\ar@<2pt>[ll]^{\text{loc. cpt.}}\ar[dd]_{\text{loc. precpt.}}\ar@<2pt>@(r,u)[ddrr]\\
\\
 &  & \PCPT\ar@<2pt>@(l,d)[uull]^{\text{vN-complete}}\ar@<2pt>[rr]\ar@<2pt>@(d,l)[ddrr]^{\text{u-I}} &  & \PNS\ar@<2pt>[ll]^{\substack{\text{C-completion}\\
\text{is loc. cpt.}
}
}\ar@<2pt>@(u,r)[uull]^{\text{C-complete}}\ar[dd]_{\text{weak u-II}}\ar@<2pt>@(r,u)[ddrr]\\
\\
 &  &  &  & \FIN\ar@<2pt>@(l,d)[uull]^{\text{preproper}}\ar@<2pt>[rr] &  & \ns{X}\ar@<2pt>[ll]^{\text{self-bdd}}\ar@<2pt>@(u,r)[uull]^{\text{precpt.}}
}
$}\caption{\label{fig:Structure-of-nonstandard-points}Each arrow $A\protect\overset{P}{\protect\longrightarrow}B$
(except for $\FIN\protect\dashrightarrow\CPT$) indicates that the
inclusion $A\subseteq B$ holds if and only if the space has property
$P$. The broken line arrow $\FIN\protect\dashrightarrow\CPT$ indicates
that the inclusion $\FIN\subseteq\CPT$ is equivalent to properness
under closure-stability.}
\end{figure}

First of all, we notice that some of the inclusions hold without any
extra condition.
\begin{fact}[{Corollary to \citep[Theorem 4.1.13]{Rob66}}]
\label{fact:CPT-sub-NS}The inclusions $X\subseteq\CPT\left(X\right)\subseteq\NS\left(X\right)$
hold for all standard topological spaces $X$.
\end{fact}

\begin{fact}
The inclusion $X\subseteq\FIN\left(X\right)$ holds for all standard
(pre)bornological spaces $X$.
\end{fact}

\begin{fact}[{Corollary to \citep[Theorem 8.4.34]{SL76}}]
\label{fact:NS-sub-PNS-and-CPT-sub-PCPT-sub-PNS}The inclusions $\NS\left(X\right)\subseteq\PNS\left(X\right)$
and $\CPT\left(X\right)\subseteq\PCPT\left(X\right)\subseteq\PNS\left(X\right)$
hold for all standard uniform spaces $X$.
\end{fact}

\subsection{Hybrid spaces and compatibility conditions}

It is often the case that a space is equipped with both small-scale
and large-scale structures. Typically, a metric space is equipped
with small-scale structures (such as the metric topology and the metric
uniformity) and large-scale ones (such as the bounded bornology and
the bounded coarse structure). Generalising this situation, we introduce
the notion of hybrid spaces.
\begin{defn}[Standard]
A \emph{tb-space} (topological-bornological space) is a set equipped
with topology and bornology. An \emph{ub-space} (uniform-bornological
space) is a set equipped with uniformity and bornology. The notions
of \emph{tc} (topological-coarse) and \emph{uc} (uniform-coarse) are
defined in a similar fashion. We call those spaces \emph{hybrid spaces}.
\end{defn}

Having two structures at hand, it is natural to consider compatibility
conditions between them.
\begin{defn}[Standard]
Let $X$ be a set, $\mathcal{T}_{X}$ a topology, $\mathcal{U}_{X}$
a uniformity, and $\mathcal{B}_{X}$ a bornology on $X$. We say that
\begin{enumerate}
\item $\mathcal{T}_{X}$ and $\mathcal{B}_{X}$ are \emph{I-compatible}
if every compact set is bounded;
\item $\mathcal{T}_{X}$ and $\mathcal{B}_{X}$ are \emph{II-compatible}
if each point has a bounded neighbourhood;
\item $\mathcal{T}_{X}$ and $\mathcal{B}_{X}$ are \emph{III-compatible}
if each bounded set has a bounded neighbourhood;
\item $\mathcal{T}_{X}$ and $\mathcal{B}_{X}$ are \emph{closure-stable}
if the (topological) closure of every bounded set is bounded;
\item $\mathcal{U}_{X}$ and $\mathcal{B}_{X}$ are \emph{u-I-compatible}
if every precompact set is bounded;
\item $\mathcal{U}_{X}$ and $\mathcal{B}_{X}$ are \emph{u-II-compatible}
if there is a $U\in\mathcal{U}_{X}$ such that $U\left[x\right]\in\mathcal{B}_{X}$
for all $x\in X$;
\item $\mathcal{U}_{X}$ and $\mathcal{B}_{X}$ are \emph{u-III-compatible}
if there is a $U\in\mathcal{U}_{X}$ such that $U\left[B\right]\in\mathcal{B}_{X}$
for all $B\in\mathcal{B}_{X}$.
\end{enumerate}
\end{defn}

\begin{rem}
Bunke and Engel \citep{BE16} use the following stronger compatibility
condition for a tb-space: $\mathcal{T}_{X}$ and $\mathcal{B}_{X}$
are called \emph{compatible} if every bounded set has a bounded neighbourhood
and a bounded closure. In our terminology, Bunke\textendash Engel's
condition is the conjunction of III-compatibility and closure-stability. 
\end{rem}

\begin{example}
Every metric space satisfies all of the compatibility conditions including
the weak u-II-compatibility (defined later).
\end{example}

\begin{example}
Let $X$ be a topological space. Obviously the compact bornology
\[
\mathcal{P}_{c}\left(X\right)=\set{A\subseteq X|A\text{ is contained in some compact set}}
\]
is I-compatible with the topology of $X$.
\end{example}

\begin{example}
Let $X$ be a uniform space. The precompact bornology
\[
\mathcal{P}_{pc}\left(X\right)=\set{A\subseteq X|A\text{ is precompact}}
\]
is u-I-compatible with the uniformity of $X$.
\end{example}

\begin{prop}[Standard]
\begin{enumerate}
\item III-compatibility implies II-compatibility.
\item u-III-compatibility implies u-II-compatibility.
\item u-I-compatibility implies I-compatibility (with respect to the induced
topology).
\item u-II-compatibility implies II-compatibility.
\item u-III-compatibility implies III-compatibility and closure-stability.
\end{enumerate}
\end{prop}

\begin{proof}
Trivial.
\end{proof}
\begin{rem}
There is an ub-space that is u-II-compatible but not closure-stable.
Consider the real line $\mathbb{R}$ with the usual uniformity. The
family
\[
\mathcal{F}_{\mathbb{R}}=\set{A\subseteq\mathbb{R}|A\text{ is contained in some measurable set with finite measure}}
\]
forms a bornology on $\mathbb{R}$. Each point $x\in\mathbb{R}$ has
a $\mathcal{F}_{\mathbb{R}}$-bounded neighbourhood of the form $\left(x-1,x+1\right)$,
so $\mathcal{F}_{\mathbb{R}}$ is u-II-compatible with the uniformity.
The set $\mathbb{Q}$ of rational numbers has measure $0$, so it
is $\mathcal{F}_{\mathbb{R}}$-bounded. However, the closure $\cl_{\mathbb{R}}\mathbb{Q}=\mathbb{R}$
is not $\mathcal{F}_{\mathbb{R}}$-bounded. Hence $\mathcal{F}_{\mathbb{R}}$
is not closure-stable.
\end{rem}

The implications II$\Rightarrow$I and u-II$\Rightarrow$u-I are non-trivial,
and will be proved by using the following nonstandard characterisations.
They are of the form $\CPT,\NS,\PCPT,\PNS\subseteq\FIN$.
\begin{fact}[{Corollary to \citep[Proposition 2.6]{Ima19}}]
\label{fact:CPT-sub-FIN}A standard tb-space $X$ is I-compatible
if and only if $\CPT\left(X\right)\subseteq\FIN\left(X\right)$.
\end{fact}

\begin{fact}[{\citep[Theorem 2.37]{Ima19}}]
\label{fact:NS-sub-FIN}A standard tb-space $X$ is II-compatible
if and only if $\NS\left(X\right)\subseteq\FIN\left(X\right)$.
\end{fact}

\begin{cor}[Standard]
\label{cor:Every-II-compatible-tb-space-is-I-compatible}Every II-compatible
tb-space is I-compatible.
\end{cor}

\begin{proof}
Let $X$ be a standard II-compatible tb-space. By \prettyref{fact:CPT-sub-NS}
and \prettyref{fact:NS-sub-FIN}, we have $\CPT\left(X\right)\subseteq\NS\left(X\right)\subseteq\FIN\left(X\right)$.
By \prettyref{fact:CPT-sub-FIN}, $X$ is I-compatible.
\end{proof}
\begin{rem}
\label{rem:Sorgenfrey-line}There is a tb-space that is I-compatible
but not II-compatible. Consider the \emph{Sorgenfrey line} $\mathbb{R}_{l}$,
which is the real numbers $\mathbb{R}$ with the topology $\mathcal{T}_{\mathbb{R}_{l}}$
generated by the right half-open intervals $\left[a,b\right)$. It
is well-known that every compact subset of $\mathbb{R}_{l}$ is countable
(see \citep[II.51.5]{SS78}). Hence $\mathcal{T}_{\mathbb{R}_{l}}$
is I-compatible with the countable bornology $\mathcal{P}_{\aleph_{0}}\left(\mathbb{R}\right)=\set{A\subseteq\mathbb{R}|A\colon\text{countable}}$.
On the other hand, every non-empty open subset of $\mathbb{R}_{l}$
is uncountable, so there is no countable neighbourhood. Hence $\mathcal{T}_{\mathbb{R}_{l}}$
is not II-compatible with $\mathcal{P}_{\aleph_{0}}\left(\mathbb{R}\right)$.
\end{rem}

\begin{fact}[{Corollary to \citep[Proposition 2.6]{Ima19}}]
\label{fact:PCPT-sub-FIN}A standard ub-space $X$ is u-I-compatible
if and only if $\PCPT\left(X\right)\subseteq\FIN\left(X\right)$.
\end{fact}

\begin{prop}
\label{prop:PNS-sub-FIN-I}If a standard ub-space $X$ is u-II-compatible,
then $\PNS\left(X\right)\subseteq\FIN\left(X\right)$.
\end{prop}

\begin{proof}
Suppose $X$ is u-II-compatible. Fix a $U\in\mathcal{U}_{X}$ with
$U\left[x\right]\in\mathcal{B}_{X}$ for all $B\in\mathcal{B}_{X}$.
Let $y\in\PNS\left(X\right)$. By definition, there exists a (standard)
$x\in X$ such that $y\in\ns{U\left[x\right]}$. Since $U\left[x\right]\in\mathcal{B}_{X}$,
it follows that $y\in\ns{U\left[x\right]}\subseteq\FIN\left(X\right)$.
\end{proof}
\begin{cor}[Standard]
\label{cor:Every-uII-compatible-ub-space-is-uI-compatible}Every
u-II-compatible ub-space is u-I-compatible.
\end{cor}

\begin{proof}
Let $X$ be a standard u-II-compatible ub-space. By \prettyref{fact:NS-sub-PNS-and-CPT-sub-PCPT-sub-PNS}
and \prettyref{prop:PNS-sub-FIN-I}, we have $\PCPT\left(X\right)\subseteq\PNS\left(X\right)\subseteq\FIN\left(X\right)$.
By \prettyref{fact:PCPT-sub-FIN}, $X$ is u-I-compatible.
\end{proof}
In contrast to \prettyref{fact:NS-sub-FIN}, the converse of \prettyref{prop:PNS-sub-FIN-I}
is not true. In fact, the inclusion $\PNS\subseteq\FIN$ is equivalent
to the following (slightly complicated) standard property, that is
intermediate between u-I and u-II.
\begin{defn}[Standard]
We say that an ub-space $X$ is \emph{weakly u-II-compatible} if
$X$ is a subspace of some Cauchy complete II-compatible ub-space
$\bar{X}$.
\end{defn}

\begin{lem}
\label{lem:NS-of-completion}Let $X$ be a standard uniform space
and $\bar{X}$ a standard Cauchy completion of $X$. Then $\NS\left(\bar{X}\right)\cap\ns{X}=\PNS\left(X\right)$.
\end{lem}

\begin{proof}
Let $x\in\NS\left(\bar{X}\right)\cap\ns{X}$. There exists a $y\in\bar{X}$
such that $x\approx_{\bar{X}}y$. Since $X$ is dense in $\bar{X}$,
we can choose, for each $U\in\mathcal{U}_{\bar{X}}$, an $x_{U}\in X$
such that $y\in U\left[x_{U}\right]$. For each $U\in\mathcal{U}_{\bar{X}}$,
since $x\in\ns{U}\left[y\right]$, we have that $x\in\ns{\left(U\circ U\right)}\left[x_{U}\right]$.
For each $V\in\mathcal{U}_{X}$, find an $U\in\mathcal{U}_{\bar{X}}$
so that $U\circ U\restriction X\subseteq V$, then $x\in\ns{V}\left[x_{U}\right]$.
Therefore $x\in\PNS\left(X\right)$.

Conversely, let $x\in\PNS\left(X\right)$. For each $U\in\mathcal{U}_{X}$,
there exists a (standard) $x_{U}\in X$ such that $x\in\ns{U}\left[x_{U}\right]$.
Notice that the family $\set{x_{U}}_{U\in\mathcal{U}_{X}}$ forms
a net in $X$ with respect to the directed set $\left(\mathcal{U}_{X},\supseteq\right)$.
Moreover, $\set{x_{U}}_{U\in\mathcal{U}_{X}}$ is Cauchy, because
$x_{U}\mathrel{\ns{U}}x\mathrel{\ns{V}^{-1}}x_{V}$ holds for all
$U,V\in\mathcal{U}_{X}$. Since $\bar{X}$ is Cauchy complete, $\set{x_{U}}_{U\in\mathcal{U}_{X}}$
converges to some $y\in\bar{X}$. It is easy to verify that $x\approx_{\bar{X}}y$.
(For each $U\in\mathcal{U}_{\bar{X}}$, find a $V\in\mathcal{U}_{\bar{X}}$
such that $x_{V\restriction X}\in U\left[y\right]$, then $y\mathrel{\ns{U}}x_{V\restriction X}\mathrel{\ns{V}}x$.)
Hence $x\in\NS\left(\bar{X}\right)\cap\ns{X}$.
\end{proof}
\begin{thm}
\label{thm:PNS-sub-FIN-II}A standard ub-space $X$ is weakly u-II-compatible
if and only if $\PNS\left(X\right)\subseteq\FIN\left(X\right)$.
\end{thm}

\begin{proof}
Suppose $\PNS\left(X\right)\subseteq\FIN\left(X\right)$. Let $\bar{X}$
be a standard Cauchy completion of $X$. We introduce a bornology
on $\bar{X}$ as follows:
\[
\mathcal{B}_{\bar{X}}=\set{A\subseteq\bar{X}|A\cap X\in\mathcal{B}_{X}}=\set{A\cup B|A\subseteq\bar{X}\setminus X\text{ and }B\in\mathcal{B}_{X}}.
\]
Clearly the restriction $\mathcal{B}_{\bar{X}}\restriction X$ coincides
with $\mathcal{B}_{X}$, i.e., $\FIN\left(X\right)=\FIN\left(\bar{X}\right)\cap\ns{X}$
(see also \citep[Example 2.18]{Ima19}). We also observe that $\bar{X}\setminus X\in\mathcal{B}_{\bar{X}}$.
Let $x\in\NS\left(\bar{X}\right)$. If $x\notin\ns{X}$, then $x\in\ns{\left(\bar{X}\setminus X\right)}$
by transfer, so $x\in\FIN\left(\bar{X}\right)$. If $x\in\ns{X}$,
then $x\in\PNS\left(X\right)$ by \prettyref{lem:NS-of-completion},
so $x\in\FIN\left(\bar{X}\right)$. Hence $\NS\left(\bar{X}\right)\subseteq\FIN\left(\bar{X}\right)$.
By \prettyref{fact:NS-sub-FIN}, $\bar{X}$ is II-compatible.

Conversely, suppose $X$ is weakly u-II-compatible, i.e., there exists
a (standard) Cauchy complete II-compatible ub-extension $\bar{X}$
of $X$. By \prettyref{lem:NS-of-completion} and \prettyref{fact:NS-sub-FIN},
$\PNS\left(X\right)=\NS\left(\bar{X}\right)\cap\ns{X}\subseteq\FIN\left(\bar{X}\right)\cap\ns{X}=\FIN\left(X\right)$.
\end{proof}
The following corollaries can be proved similarly to \prettyref{cor:Every-II-compatible-tb-space-is-I-compatible}
and \prettyref{cor:Every-uII-compatible-ub-space-is-uI-compatible}.
\begin{cor}[Standard]
\begin{enumerate}
\item u-II-compatibility implies weak u-II-compatibility.
\item u-II-compatibility implies u-I-compatibility.
\item Weak u-II-compatibility implies II-compatibility.
\end{enumerate}
\end{cor}

We have shown the following implications among the compatibility conditions.
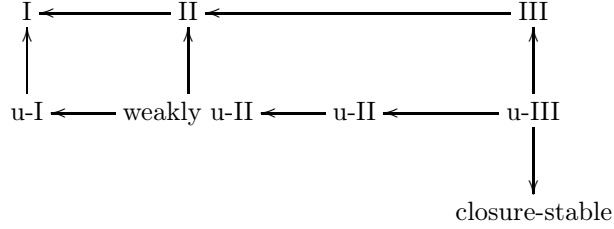
\[
\xymatrix{\text{I} & \text{II}\ar[l] &  & \text{III}\ar[ll]\\
\text{u-I}\ar[u] & \text{weakly u-II}\ar[u]\ar[l] & \text{u-II}\ar[l] & \text{u-III}\ar[d]\ar[u]\ar[l]\\
 &  &  & \text{closure-stable}
}
\]
Some implications can be reversed under certain conditions, as we
will see later (\prettyref{cor:reverse-compatibility-conditions}).

u-III-compatibility has a similar nonstandard characterisation. To
state the characterisation, we introduce another class of nonstandard
points, called nearfinite points.
\begin{defn}
Let $X$ be a standard ub-space. The elements of $\NF\left(X\right)=\mu_{X}^{u}\left(\FIN\left(X\right)\right)$
are called \emph{nearfinite points}.
\end{defn}

The inclusion $\FIN\subseteq\NF$ is obvious. The reverse inclusion
$\NF\subseteq\FIN$ characterises u-III-compatibility.
\begin{thm}
A standard ub-space $X$ is u-III-compatible if and only if $\NF\left(X\right)\subseteq\FIN\left(X\right)$.
\end{thm}

\begin{proof}
Suppose $X$ is u-III-compatible. Let $x\in\NF\left(X\right)$, i.e.,
$x\in\mu_{X}^{u}\left(y\right)$ holds for some $y\in\FIN\left(X\right)$.
Let $B\in\mathcal{B}_{X}$ such that $y\in\ns{B}$. By the u-III-compatibility,
there exists a $U\in\mathcal{U}_{X}$ such that $U\left[B\right]\in\mathcal{B}_{X}$.
Hence $x\in\mu_{X}^{u}\left(y\right)\subseteq\ns{U}\left[y\right]\subseteq\ns{U}\left[\ns{B}\right]\subseteq\FIN\left(X\right)$.

Conversely, suppose $\NF\left(X\right)\subseteq\FIN\left(X\right)$.
Let $A\in\mathcal{B}_{X}$. By \prettyref{lem:General-Approximation-Lemma},
there exist a $B\in\ns{\mathcal{B}_{X}}$ and a $U\in\ns{\mathcal{U}_{X}}$
such that $\FIN\left(X\right)\subseteq B$ and $U\subseteq{\approx_{X}}$.
Then $U\left[\ns{A}\right]\subseteq\mu_{X}^{u}\left(\ns{A}\right)\subseteq\NF\left(X\right)\subseteq\FIN\left(X\right)\subseteq B$,
so $U\left[\ns{A}\right]\in\ns{\mathcal{B}_{X}}$. Hence $V\left[A\right]\in\mathcal{B}_{X}$
for some $V\in\mathcal{U}_{X}$ by transfer.
\end{proof}
Recall that I-, II-, u-I-, weak u-II-, and u-III-compatibility conditions
are characterised by the inclusions $\CPT,\NS,\PCPT,\PNS,\NF\subseteq\FIN$.
(On the other hand, III- and u-II-compatibility have not been characterised
in terms of inclusion. To do this, it is necessary to find appropriate
properties of nonstandard points like nearfiniteness.) Interestingly,
$\CPT$, $\NS$, $\PCPT$ and $\PNS$ are purely small-scale notions
and $\FIN$ is a purely large-scale notion, while $\NF$ is a hybrid-scale
notion involving both uniformity and bornology. $\NF$ is an interesting
class of nonstandard points, but beyond the scope of this paper.

\subsection{Characterisations of single-scale properties}

We next consider single-scale properties characterised by the inclusion
relations among classes of nonstandard points.
\begin{fact}[{\citep[Theorem 2.11.2]{Rob66}}]
\label{fact:starX-sub-X}A standard set $X$ is finite if and only
if $\ns{X}\subseteq X$.
\end{fact}

\begin{fact}
\label{fact:CPT-sub-X}A standard topological space $X$ has no infinite
compact subset if and only if $\CPT\left(X\right)\subseteq X$.
\end{fact}

\begin{fact}
\label{fact:PCPT-sub-X}A standard uniform space $X$ has no infinite
precompact subset if and only if $\PCPT\left(X\right)\subseteq X$.
\end{fact}

\begin{fact}[{Corollary to \citep[Proposition 2.6]{Ima19}}]
\label{fact:FIN-sub-X}A standard bornological space $X$ has no
infinite bounded subset if and only if $\FIN\left(X\right)\subseteq X$.
\end{fact}

\prettyref{fact:starX-sub-X}, \prettyref{fact:CPT-sub-X} and \prettyref{fact:PCPT-sub-X}
can be considered as special cases of \prettyref{fact:FIN-sub-X},
where $X$ is equipped with the maximal bornology \citep[Example 2.13]{Ima19},
the compact bornology \citep[Example 2.16]{Ima19}, and the precompact
bornology, respectively. Similarly, the following two characterisations
can be considered as special cases of \prettyref{fact:NS-sub-FIN}.
\begin{fact}[{\citep[Theorem 3.7.1]{Lux69}}]
\label{fact:NS-sub-CPT}A standard topological space $X$ is locally
compact if and only if $\NS\left(X\right)\subseteq\CPT\left(X\right)$.
\end{fact}

\begin{proof}
The space $X$ is equipped with the compact bornology. Then $\FIN\left(X\right)=\CPT\left(X\right)$.
Obviously $X$ is locally compact if and only if the topology and
the bornology of $X$ are II-compatible. It is also equivalent to
$\NS\left(X\right)\subseteq\CPT\left(X\right)$ by \prettyref{fact:NS-sub-FIN}.
See also \citep[Corollary 2.38]{Ima19}.
\end{proof}
\begin{fact}[{\citep[Theorem 8.4.37]{SL76}}]
\label{fact:NS-sub-PCPT}A standard uniform space $X$ is locally
precompact if and only if $\NS\left(X\right)\subseteq\PCPT\left(X\right)$.
\end{fact}

\begin{proof}
Similarly to \prettyref{fact:NS-sub-CPT}, $X$ is locally precompact
$\iff$ the (induced) topology and the precompact bornology of $X$
are II-compatible $\iff$ $\NS\left(X\right)\subseteq\FIN\left(X\right)=\PCPT\left(X\right)$
by \prettyref{fact:NS-sub-FIN}.
\end{proof}
\begin{defn}[{Standard; \citep[Definition IV]{Neu35}}]
A uniform space is said to be \emph{von Neumann complete} if every
closed precompact subset is compact.
\end{defn}

\begin{thm}
\label{thm:PCPT-sub-CPT}A standard uniform space $X$ is von Neumann
complete if and only if $\PCPT\left(X\right)\subseteq\CPT\left(X\right)$.
\end{thm}

\begin{proof}
Suppose $X$ is von Neumann complete. Let $P$ be a precompact subset
of $X$. The closure $\cl_{X}P$ is closed and precompact, so it is
compact. Thus
\[
\PCPT\left(X\right)=\bigcup_{P\colon\text{precompact}}\ns{P}\subseteq\bigcup_{P\colon\text{precompact}}\ns{\left(\cl_{X}P\right)}\subseteq\bigcup_{K\colon\text{compact}}\ns{K}=\CPT\left(X\right).
\]

Conversely, suppose $\PCPT\left(X\right)\subseteq\CPT\left(X\right)$.
Let $P$ be a closed precompact subset of $X$. Then $\ns{P}\subseteq\PCPT\left(X\right)\subseteq\CPT\left(X\right)$.
By \citep[Proposition 2.6]{Ima19}, there exists a compact subset
$K$ of $X$ such that $P\subseteq K$. Since $P$ is a closed subset
of the compact set $K$, $P$ is compact.
\end{proof}
\begin{fact}[{\citep[Theorem 8.4.37]{SL76}}]
\label{fact:PNS-sub-PCPT}A standard uniform space $X$ has a locally
compact Cauchy completion if and only if $\PNS\left(X\right)\subseteq\PCPT\left(X\right)$.
\end{fact}

\begin{fact}[{\citep[Theorem 3.14.1]{Lux69}}]
\label{fact:PNS-sub-NS}A standard uniform space $X$ is Cauchy complete
if and only if $\PNS\left(X\right)\subseteq\NS\left(X\right)$.
\end{fact}

\begin{fact}[{\citep[Theorem 4.1.13]{Rob66}}]
\label{fact:starX-sub-NS}A standard topological space $X$ is compact
if and only if $\ns{X}\subseteq\NS\left(X\right)$.
\end{fact}

\begin{fact}[{\citep[Theorem 3.13.1]{Lux69}}]
\label{fact:starX-sub-PNS}A standard uniform space $X$ is precompact
if and only if $\ns{X}\subseteq\PNS\left(X\right)$.
\end{fact}

\prettyref{fact:PNS-sub-NS} to \prettyref{fact:starX-sub-PNS} make
easy to prove the following well-known theorem in elementary topology.
Since we will use a similar technique later (\prettyref{cor:compact-loc-cpt-and-preproper}),
we here review a nonstandard proof in \citep[Theorem 8.4.35]{SL76}.
\begin{cor}[Standard]
A uniform space is compact if and only if it is Cauchy complete and
precompact.
\end{cor}

\begin{proof}
Let $X$ be a standard uniform space. By \prettyref{fact:NS-sub-PNS-and-CPT-sub-PCPT-sub-PNS},
$\NS\left(X\right)\subseteq\PNS\left(X\right)\subseteq\ns{X}$. If
$X$ is compact, then $\ns{X}\subseteq\NS\left(X\right)$ by \prettyref{fact:starX-sub-NS},
so $\NS\left(X\right)=\PNS\left(X\right)=\ns{X}$. The first equality
implies the Cauchy completeness by \prettyref{fact:PNS-sub-NS}; and
the second one implies the precompactness by \prettyref{fact:starX-sub-PNS}.

Conversely, if $X$ is Cauchy complete and precompact, then $\ns{X}\subseteq\PNS\left(X\right)\subseteq\NS\left(X\right)$,
so $X$ is compact by \prettyref{fact:starX-sub-NS}.
\end{proof}
These characterisations have some consequences on the compatibility
conditions.
\begin{cor}[Standard]
\label{cor:reverse-compatibility-conditions}
\begin{enumerate}
\item \label{enu:compatibility-conditions-reverse-1}Every locally compact
I-compatible tb-space is II-compatible.
\item \label{enu:compatibility-conditions-reverse-2}Every u-I-compatible
ub-space having a locally compact Cauchy completion is weakly u-II-compatible.
\item Every locally precompact u-I-compatible ub-space is II-compatible.
\item Every von Neumann complete I-compatible ub-space is u-I-compatible.
\item Every Cauchy complete II-compatible ub-space is weakly u-II-compatible.
\end{enumerate}
\end{cor}

\begin{proof}
Let us only prove \eqref{enu:compatibility-conditions-reverse-1}
and \eqref{enu:compatibility-conditions-reverse-2}. The others can
be proved in a similar way.

Let $X$ be a standard locally compact I-compatible tb-space. By \prettyref{fact:NS-sub-CPT}
and \prettyref{fact:CPT-sub-FIN}, $\NS\left(X\right)\subseteq\CPT\left(X\right)\subseteq\FIN\left(X\right)$.
By \prettyref{fact:NS-sub-FIN}, $X$ is II-compatible.

Let $X$ be a standard u-I-compatible ub-space. Suppose $X$ has a
locally compact Cauchy completion. By \prettyref{fact:PNS-sub-PCPT}
and \prettyref{fact:PCPT-sub-FIN}, $\PNS\left(X\right)\subseteq\PCPT\left(X\right)\subseteq\FIN\left(X\right)$.
Hence $X$ is weakly u-II-compatible by \prettyref{thm:PNS-sub-FIN-II}.
\end{proof}
We have shown the following (reverse) implications among the compatibility
conditions.
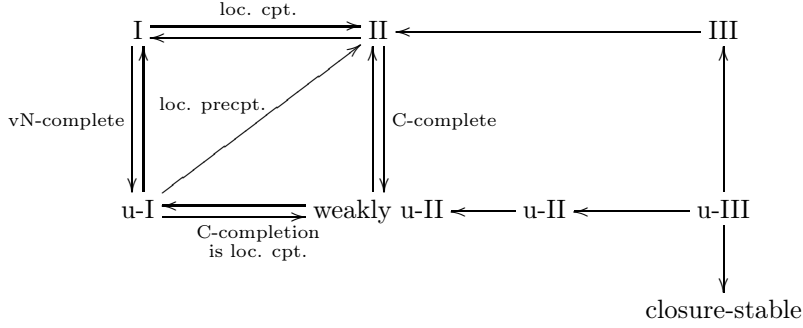
\[
\xymatrix{\text{I}\ar@<-0.5ex>[dd]_{\text{vN-complete}}\ar@<0.5ex>[rr]^{\text{loc. cpt.}} &  & \text{II}\ar@<0.5ex>[dd]^{\text{C-complete}}\ar@<0.5ex>[ll] &  & \text{III}\ar[ll]\\
\\
\text{u-I}\ar@<-0.5ex>[uu]\ar@<-0.5ex>[rr]_{\substack{\text{C-completion}\\
\text{is loc. cpt.}
}
}\ar[uurr]^{\text{loc. precpt.}} &  & \text{weakly u-II}\ar@<0.5ex>[uu]\ar@<-0.5ex>[ll] & \text{u-II}\ar[l] & \text{u-III}\ar[d]\ar[uu]\ar[l]\\
 &  &  &  & \text{closure-stable}
}
\]

\subsection{Characterisations of properness and preproperness}

We shall show that the inclusions $\FIN\subseteq\NS$ and $\FIN\subseteq\PNS$
characterise properness and preproperness, respectively. These characterisations
are known for the case of metric spaces (see e.g. \citep[Theorem 5.6 of Chapter 3]{Dav05},
\citep[Proposition 10.1.25]{SL76} and \citep[Definition 6.6.4 and Answer to Exercise 6.6.1]{DD95}).
In our general setting, the former characterisation requires the closure
stability condition.
\begin{defn}[Standard]
A tb-space is said to be \emph{proper} if every bounded closed set
is compact.
\end{defn}

\begin{thm}
\label{thm:FIN-sub-CPT}Let $X$ be a standard closure-stable tb-space.
The following are equivalent:
\begin{enumerate}
\item \label{enu:FIN-sub-CPT-cond-1}$X$ is proper;
\item \label{enu:FIN-sub-CPT-cond-2}$\FIN\left(X\right)\subseteq\CPT\left(X\right)$;
\item \label{enu:FIN-sub-CPT-cond-3}$\FIN\left(X\right)\subseteq\NS\left(X\right)$.
\end{enumerate}
\end{thm}

\begin{proof}
\eqref{enu:FIN-sub-CPT-cond-1}$\Rightarrow$\eqref{enu:FIN-sub-CPT-cond-2}:
$\FIN\left(X\right)=\bigcup_{B\colon\text{bounded}}\ns{B}\subseteq\bigcup_{B\colon\text{bounded}}\ns{\left(\cl_{X}\left(B\right)\right)}\subseteq\bigcup_{K\colon\text{compact}}\ns{K}=\CPT\left(X\right)$.

\eqref{enu:FIN-sub-CPT-cond-2}$\Rightarrow$\eqref{enu:FIN-sub-CPT-cond-3}:
Trivial.

\eqref{enu:FIN-sub-CPT-cond-3}$\Rightarrow$\eqref{enu:FIN-sub-CPT-cond-1}:
Let $B$ be a bounded closed set of $X$. By \prettyref{fact:starX-sub-NS},
we only need to show that $\NS\left(B\right)=\ns{B}$. Let $x\in\ns{B}$.
Since $\ns{B}\subseteq\FIN\left(X\right)\subseteq\NS\left(X\right)$,
$x\in\NS\left(X\right)$. One can find a $y\in X$ so that $x\in\mu_{X}\left(y\right)$,
i.e., $\mu_{X}\left(y\right)\cap\ns{B}$ is non-empty. By the nonstandard
characterisation of closedness \citep[Theorem 4.1.5]{Rob66}, we have
that $y\in B$. Therefore $x\in\NS\left(B\right)$.
\end{proof}
\begin{defn}[Standard]
An ub-space is said to be \emph{preproper} if every bounded set is
precompact.
\end{defn}

\begin{thm}
\label{thm:FIN-sub-PCPT}Let $X$ be a standard ub-space. The following
are equivalent:
\begin{enumerate}
\item \label{enu:FIN-sub-PCPT-cond-1}$X$ is preproper;
\item \label{enu:FIN-sub-PCPT-cond-2}$\FIN\left(X\right)\subseteq\PCPT\left(X\right)$;
\item \label{enu:FIN-sub-PCPT-cond-3}$\FIN\left(X\right)\subseteq\PNS\left(X\right)$.
\end{enumerate}
\end{thm}

\begin{proof}
\eqref{enu:FIN-sub-PCPT-cond-1}$\Rightarrow$\eqref{enu:FIN-sub-PCPT-cond-2}:
$\FIN\left(X\right)=\bigcup_{B\colon\text{bounded}}\ns{B}\subseteq\bigcup_{P\colon\text{precompact}}\ns{P}=\PCPT\left(X\right)$.

\eqref{enu:FIN-sub-PCPT-cond-2}$\Rightarrow$\eqref{enu:FIN-sub-PCPT-cond-3}:
Trivial.

\eqref{enu:FIN-sub-PCPT-cond-3}$\Rightarrow$\eqref{enu:FIN-sub-PCPT-cond-1}:
Let $B$ be a bounded set. Then $\ns{B}\subseteq\FIN\left(X\right)\subseteq\PNS\left(X\right)$.
By \citep[Theorem 8.4.34]{SL76}, $B$ is precompact.
\end{proof}
The results we have obtained so far are summarised in \prettyref{fig:Structure-of-nonstandard-points}.
Just by looking at the figure, we may produce various (complex) statements
on general topology. An example is as follows.
\begin{cor}[Standard]
\label{cor:compact-loc-cpt-and-preproper}A closure-stable weakly
u-II-compatible ub-space (such as a metric space) is proper if and
only if it is Cauchy complete and preproper. Such spaces are locally
compact.
\end{cor}

\begin{proof}
Let $X$ be a standard closure-stable u-II-compatible ub-space. By
\prettyref{fact:CPT-sub-NS}, \prettyref{fact:NS-sub-PNS-and-CPT-sub-PCPT-sub-PNS}
and \prettyref{thm:PNS-sub-FIN-II}, $\CPT\left(X\right)\subseteq\NS\left(X\right)\subseteq\PNS\left(X\right)\subseteq\FIN\left(X\right)$.
If $X$ is proper, then $\FIN\left(X\right)\subseteq\CPT\left(X\right)$
by \prettyref{thm:FIN-sub-CPT}, so $\CPT\left(X\right)=\NS\left(X\right)=\PNS\left(X\right)=\FIN\left(X\right)$.
The first equality implies the local compactness by \prettyref{fact:NS-sub-CPT};
the second equality implies the Cauchy completeness by \prettyref{fact:PNS-sub-NS};
and the third equality implies the preproperness by \prettyref{thm:FIN-sub-PCPT}.

Conversely, if $X$ is Cauchy complete and preproper, then $\NS\left(X\right)\supseteq\PNS\left(X\right)\supseteq\FIN\left(X\right)$
by \prettyref{fact:PNS-sub-NS} and \prettyref{thm:FIN-sub-PCPT},
so $X$ is proper by \prettyref{thm:FIN-sub-CPT}.
\end{proof}

\section{\label{sec:Large-scale-structures-on-ns-ext}Large-scale structures
on nonstandard extensions}

In \prettyref{sec:Nonstandard-points}, we studied the structure of
a standard space $X$ by using the nonstandard extension $\ns{X}$
as an auxiliary tool. In the present section, we focus, in contrast,
on the structure of the nonstandard space $\ns{X}$ itself. For this
purpose, we introduce two large-scale structures on $\ns{X}$: S-prebornology
and S-coarse structure.

\subsection{S-prebornologies}
\begin{prop}[S-prebornology]
Given a standard prebornological space $\left(X,\mathcal{B}_{X}\right)$,
the family
\[
\ss{\mathcal{B}_{X}}=\set{A\subseteq\FIN\left(X\right)|A\subseteq\ns{B}\text{ for some }B\in\mathcal{B}_{X}}
\]
is a prebornology on $\FIN\left(X\right)$.
\end{prop}

\begin{proof}
$\ss{\mathcal{B}_{X}}$ is generated by $\ssig{\mathcal{B}_{X}}=\set{\ns{B}|B\in\mathcal{B}_{X}}$.
It suffices to prove that $\ssig{\mathcal{B}_{X}}$ covers $\FIN\left(X\right)$
and is closed under finite non-disjoint unions. The former is trivial
by the definition of $\FIN\left(X\right)=\bigcup_{B\in\mathcal{B}_{X}}\ns{B}=\bigcup\ssig{\mathcal{B}_{X}}$.
The latter follows from the transfer principle.
\end{proof}
\begin{notation}
We denote the prebornological space $\left(\FIN\left(X\right),\ss{\mathcal{B}_{X}}\right)$
by $SX$.
\end{notation}

\begin{lem}
\label{lem:S-bounded-sets}Let $X$ be a standard prebornological
space. For every subset $B$ of $X$, $B\in\mathcal{B}_{X}$ if and
only if $\ns{B}\in\ss{\mathcal{B}_{X}}$.
\end{lem}

\begin{proof}
The `only if' part is trivial. Suppose $\ns{B}\in\ss{\mathcal{B}_{X}}$.
There exists a $B'\in\mathcal{B}_{X}$ such that $\ns{B}\subseteq\ns{B'}$.
By transfer, $B\subseteq B'$, so $B\in\mathcal{B}_{X}$.
\end{proof}
\begin{prop}
A standard prebornological space $X$ is connected if and only if
$SX$ is connected.
\end{prop}

\begin{proof}
Suppose $X$ is connected. Let $B=\set{x_{1},\ldots,x_{n}}$ be a
finite subset of $\mathrm{FIN}\left(X\right)$. For each $i\leq n$,
choose a (standard) $B_{i}\in\mathcal{B}_{X}$ so that $x_{i}\in\ns{B_{i}}$.
Since $X$ is connected, $B_{1}\cup\cdots\cup B_{n}\in\mathcal{B}_{X}$
holds. Hence $B\in\ss{\mathcal{B}_{X}}$.

Conversely, suppose $SX$ is connected. Let $B$ be a finite subset
of $X$. Then $B$ is a finite subset of $\FIN\left(X\right)$, so
$B\in\ss{\mathcal{B}_{X}}$. Since $B=\ns{B}$ (by transfer), we have
that $B\in\mathcal{B}_{X}$ by \prettyref{lem:S-bounded-sets}. Hence
$X$ is connected.
\end{proof}
The S-prebornology construction can be extended to a functor from
the category of standard prebornological spaces to the category of
(external) prebornological spaces, where the morphisms are bornological
maps.
\begin{thm}
\label{thm:S-bornology-is-functorial}A map $f\colon X\to Y$ between
standard prebornological spaces is bornological if and only if $\ns{f}\colon SX\to SY$
is well-defined and bornological.
\end{thm}

\begin{proof}
Suppose that $f\colon X\to Y$ is bornological. By the nonstandard
characterisation of bornologicity \citep[Theorem 2.24]{Ima19}, we
have $\ns{f}\left(\FIN\left(X\right)\right)\subseteq\FIN\left(Y\right)$.
Therefore $\ns{f}\colon SX\to SY$ is well-defined. Let $B\in\ss{\mathcal{B}_{X}}$.
Choose a $B'\in\mathcal{B}_{X}$ so that $B\subseteq\ns{B'}$. Obviously,
$\ns{f}\left(B\right)\subseteq\ns{f}\left(\ns{B'}\right)=\ns{\left(f\left(B'\right)\right)}$.
Since $f$ is bornological, $f\left(B'\right)\in\mathcal{B}_{Y}$
holds. Hence $\ns{f}\left(B\right)\in\ss{\mathcal{B}_{Y}}$.

Conversely, suppose that $\ns{f}\colon SX\to SY$ is well-defined
and bornological. Let $B\in\mathcal{B}_{X}$. Then $\ns{B}\in\ss{\mathcal{B}_{X}}$,
so $\ns{\left(f\left(B\right)\right)}=\ns{f}\left(\ns{B}\right)\in\ss{\mathcal{B}_{Y}}$.
Hence $f\left(B\right)\in\mathcal{B}_{Y}$ by \prettyref{lem:S-bounded-sets}.
\end{proof}
The inclusion map $i_{X}\colon X\hookrightarrow SX$ can be considered
as a natural embedding.
\begin{prop}
\label{prop:inclusion-of-S-born}For each standard prebornological
space $X$, the inclusion map $i_{X}\colon X\hookrightarrow SX$ is
bornological and proper.
\end{prop}

\begin{proof}
Let $A\in\mathcal{B}_{X}$. Since $i_{X}\left(A\right)=A\subseteq\ns{A}\in\ss{\mathcal{B}_{X}}$,
we have that $i_{X}\left(A\right)\in\ss{\mathcal{B}_{X}}$. Next,
let $B\in\ss{\mathcal{B}_{X}}$. Choose a $C\in\mathcal{B}_{X}$ so
that $B\subseteq\ns{C}$. Then, $i_{X}^{-1}\left(B\right)\subseteq i_{X}^{-1}\left(\ns{C}\right)=C$,
where the latter equality follows from the transfer principle. Hence
$i_{X}^{-1}\left(B\right)\in\mathcal{B}_{X}$.
\end{proof}

\subsection{Prebornological ultrafilter spaces}

We first recall the connection between S-topologies and ultrafilters.
\begin{defn}[\citep{Lux69,ST00,Str72}]
Let $\left(X,\mathcal{T}_{X}\right)$ be a standard topological space.
The \emph{S-topology} on $\ns{X}$ is the topology $\ss{\mathcal{T}_{X}}$
generated by $\prescript{\sigma}{}{\mathcal{T}_{X}}=\set{\ns{U}|U\in\mathcal{T}_{X}}$.
We denote the space $\ns{X}$ together with $\ss{\mathcal{T}_{X}}$
by $S^{t}X$.
\end{defn}

\begin{rem}
The Robinson's S-topology appeared in \citep{Rob66} is different
from the above (Luxemburg's) one. Let $\left(X,d_{X}\right)$ be a
standard metric space. The Robinson's S-topology on $\ns{X}$ is generated
by $\set{\ns{B}\left(x;\varepsilon\right)|x\in\ns{X}\text{ and }\varepsilon\in\mathbb{R}_{+}}$,
while the Luxemburg's S-topology on $\ns{X}$ is generated by $\set{\ns{B\left(x;\varepsilon\right)}|x\in X\text{ and }\varepsilon\in\mathbb{R}_{+}}$.
\end{rem}

The (Luxemburg's) S-topology is non-trivial and highly complicated
in general. For instance, if $X$ is completely regular Hausdorff,
the $T_{2}$-reflection of $S^{t}X$ coincides with the Stone\textendash \v{C}ech
compactification $\beta X$ \citep[Theorem 4.2]{ST00}. More precisely,
the following connection holds (see also \citep{Lux69} and \citep{Str72}).
\begin{defn}[Standard; \citep{Sal00}]
Let $\left(X,\mathcal{T}_{X}\right)$ be a topological space. Let
$\Ult X$ be the set of all ultrafilters on $X$. The sets of the
form $\set{F\in\Ult X|U\in F}$, where $U\in\mathcal{T}_{X}$, generate
a topology on $\Ult X$. The topological space $\Ult X$ is called
the \emph{(topological) ultrafilter space} of $X$.
\end{defn}

\begin{thm}
\label{thm:StX-and-UltX}Let $X$ be a standard topological space.
For each $x\in\ns{X}$, let $F_{x}=\set{A\in\mathcal{P}\left(X\right)|x\in\ns{A}}$.
Then the map $\Phi_{X}\colon x\mapsto F_{x}$ is an open continuous
surjection from $S^{t}X$ to $\Ult X$.
\end{thm}

\begin{proof}
We first verify the well-definedness. Let $x\in\ns{X}$. If $A,B\in F_{x}$,
then $x\in\ns{A}\cap\ns{B}=\ns{\left(A\cap B\right)}$, so $A\cap B\in F_{x}$.
If $A\supseteq B\in F_{x}$, then $x\in\ns{B}\subseteq\ns{A}$, so
$A\in F_{x}$. If $A\notin F_{x}$, then $x\in\ns{\left(X\setminus A\right)}$,
so $X\setminus A\in F_{x}$. Clearly $\varnothing\notin F_{x}$. Hence
$F_{x}$ is an ultrafilter over $X$.

Let $A\subseteq X$. Then
\begin{align*}
\Phi_{X}^{-1}\left(\set{F\in\Ult X|A\in F}\right) & =\set{x\in\ns{X}|A\in F_{x}}\\
 & =\set{x\in\ns{X}|x\in\ns{A}}\\
 & =\ns{A}.
\end{align*}
Conversely, we show that $\Phi_{X}\left(\ns{A}\right)=\set{F\in\Ult X|A\in F}$.
The inclusion $\subseteq$ is trivial. Let $F\in\Ult X$ be such that
$A\in F$. Since $F$ has the finite intersection property, the intersection
$\bigcap_{B\in F}\ns{B}$ is non-empty by weak saturation. Fix an
$x\in\bigcap_{B\in F}\ns{B}\subseteq\ns{A}$. Then $F\subseteq F_{x}$,
so $F=F_{x}$ by the maximality of $F$. Hence $F\in\Phi_{X}\left(\ns{A}\right)$.

Since $\Phi_{X}\left(\ns{U}\right)=\set{F\in\Ult X|U\in F}$ and $\Phi_{X}^{-1}\left(\set{F\in\Ult X|U\in F}\right)=\ns{U}$
hold for all $U\in\mathcal{T}_{X}$, the map $\Phi_{X}$ is open and
continuous. This also implies the surjectivity: $\Phi_{X}\left(\ns{X}\right)=\set{F\in\Ult X|X\in F}=\Ult X$.
\end{proof}
As a by-product, we obtain a nonstandard construction of the ultrafilter
space.
\begin{cor}
$\Ult X\cong S^{t}X/\ker\Phi_{X}$.
\end{cor}

\begin{cor}[{Standard; \citep[Theorem 1]{Sal00}}]
$\Ult X$ is compact.
\end{cor}

\begin{proof}
By weak saturation, $S^{t}X$ is compact (see \citep[Theorem 2.3]{ST00}).
Since $\Ult X$ is the image of $S^{t}X$ by the continuous map $\Phi_{X}$,
it is compact.
\end{proof}
For instance, if $X$ is a discrete space, then $S^{t}X/\ker\Phi_{X}\cong\Ult X\cong\beta X$
\citep[Theorem 2.5.5]{Lux69}. In fact, all Hausdorff compactifications
can be obtained in a similar way. See \citep{ST11} for more details.

Next, we consider a large-scale analogue of this connection. We shall
introduce a natural prebornology on the set of $\flat$-ultrafilters.
\begin{defn}[Standard]
Let $\left(X,\mathcal{B}_{X}\right)$ be a prebornological space.
We call a filter $F$ on $X$ a \emph{$\flat$-filter} if $F\cap\mathcal{B}_{X}\neq\varnothing$.
Let $\bUlt X$ be the set of all $\flat$-ultrafilters on $X$. The
sets of the form $\set{F\in\bUlt X|B\in F}$, where $B\in\mathcal{B}_{X}$,
generate a prebornology on $\bUlt X$. We call the prebornological
space $\bUlt X$ the \emph{prebornological ultrafilter space} of $X$.
\end{defn}

\begin{rem}
The sets of the form $\set{F\in\Ult X|B\in F}$, where $B\in\mathcal{B}_{X}$,
cover the set $\bUlt X$, while they do not cover $\Ult X$ except
for the case where $X$ is bounded in itself. Because of this, it
is reasonable to restrict the underlying set to $\bUlt X$.
\end{rem}

\begin{thm}
\label{thm:SX-and-bUltX}Let $X$ be a standard prebornological space.
For each $x\in\FIN\left(X\right)$, let $F_{x}=\set{A\in\mathcal{P}\left(X\right)|x\in\ns{A}}$.
Then the map $\Psi_{X}\colon x\mapsto F_{x}$ is a proper bornological
surjection from $SX$ to $\bUlt X$.
\end{thm}

\begin{proof}
We first verify the well-definedness. Let $x\in\FIN\left(X\right)$.
As already shown in the proof of \prettyref{thm:StX-and-UltX}, $F_{x}$
is an ultrafilter over $X$. Since $x$ is finite, we can find a $B\in\mathcal{B}_{X}$
so that $x\in\ns{B}$. Then $B\in F_{x}$. Therefore $F_{x}$ is a
$\flat$-filter.

Similarly to the proof of \prettyref{thm:StX-and-UltX}, we can prove
that $\Psi_{X}^{-1}\left(\set{F\in\bUlt X|B\in F}\right)=\ns{B}$
and $\Psi_{X}\left(\ns{B}\right)=\set{F\in\bUlt X|B\in F}$ hold for
all $B\in\mathcal{B}_{X}$, and therefore $\Psi_{X}$ is bornological,
proper and surjective.
\end{proof}
\begin{cor}
$\bUlt X\cong SX/\ker\Psi_{X}$.
\end{cor}

It is known that every topological space $X$ is patch-densely embeddable
into $\Ult X$ (through the map $X\hookrightarrow S^{t}X\twoheadrightarrow\Ult X$)
\citep[Proposition 2]{Sal00}. Its prebornological analogue can be
stated as follows.
\begin{defn}[Standard]
A subset $A$ of a prebornological space $X$ is said to be \emph{bornologically
dense} (abbreviated as \emph{B-dense}) if $A$ has a non-empty intersection
with each connected component of $X$.
\end{defn}

\begin{thm}[Standard]
Every prebornological space $X$ is B-densely embeddable into $\bUlt X$.
\end{thm}

\begin{proof}
By \prettyref{prop:inclusion-of-S-born} and \prettyref{thm:SX-and-bUltX},
the inclusion map $i_{X}\colon X\hookrightarrow SX$ and the map $\Psi_{X}\colon SX\twoheadrightarrow\bUlt X$
defined above are proper bornological, so is the the composition $j_{X}=\Psi_{X}\circ i_{X}\colon X\to\bUlt X$.
Notice that $j_{X}\left(x\right)$ is the principal ultrafilter $\set{A\in\mathcal{P}\left(X\right)|x\in A}$
for each $x\in X$. Hence $j_{X}$ is injective.

Let $F\in\bUlt X$. Let $B\in\mathcal{B}_{X}$ be such that $B\in F$,
and pick $x\in B$. Then $F$ and $j_{X}\left(x\right)$ are included
in the same bounded subset $\set{G\in\bUlt X|B\in G}$ of $\bUlt X$.
Hence the image of $j_{X}$ is B-dense in $\bUlt X$.
\end{proof}
Finally, we discuss compatibility issues. Given a tb-space $X$, $\bUlt X$
can be regarded as a tb-space by considering the subspace topology
in $\Ult X$. Similarly, given a standard tb-space $X$, $SX$ can
be be regarded as a tb-space by considering the subspace topology
in $S^{t}X$.
\begin{thm}
\label{thm:preservation-compatibility-by-S}Let $X$ be a standard
tb-space. If $X$ is III-compatible, then so is $SX$.
\end{thm}

\begin{proof}
Let $A\in\ss{\mathcal{B}_{X}}$, i.e., there is a $B\in\mathcal{B}_{X}$
such that $A\subseteq\ns{B}$. By the III-compatibility, there exists
an $N\in\mathcal{B}_{X}\cap\mathcal{T}_{X}$ such that $B\subseteq N$.
By transfer, $A\subseteq\ns{B}\subseteq\ns{N}$ and $\ns{N}\in\left(\ss{\mathcal{T}_{X}}\restriction\FIN\left(X\right)\right)\cap\ss{\mathcal{B}_{X}}$.
Hence $X$ is III-compatible.
\end{proof}
\begin{lem}[Standard]
Let $q\colon X\to Y$ be an open proper bornological surjection between
tb-spaces. If $X$ is III-compatible, then so is $Y$.
\end{lem}

\begin{proof}[Proof (Standard)]
Let $B\in\mathcal{B}_{Y}$. Since $q$ is proper, $q^{-1}\left(B\right)\in\mathcal{B}_{X}$,
so there is an $N\in\mathcal{T}_{X}\cap\mathcal{B}_{X}$ so that $q^{-1}\left(B\right)\subseteq N$.
Since $q$ is open, bornological and surjective, $B=q\left(q^{-1}\left(B\right)\right)\subseteq q\left(N\right)\in\mathcal{T}_{Y}\cap\mathcal{B}_{Y}$.
\end{proof}
Combining this lemma with \prettyref{thm:preservation-compatibility-by-S}
yields the following preservation result.
\begin{cor}[Standard]
Let $X$ be a tb-space. If $X$ is III-compatible, then so is $\bUlt X$.
\end{cor}

\subsection{S-coarse structures}

We next consider the coarse counterpart of S-prebornology. 
\begin{prop}[S-coarse structure]
Given a standard coarse space $\left(X,\mathcal{C}_{X}\right)$,
the family
\[
\ss{\mathcal{C}_{X}}=\set{E\subseteq\ns{X}\times\ns{X}|E\subseteq\ns{F}\text{ for some }F\in\mathcal{C}_{X}}
\]
is a coarse structure on $\ns{X}$.
\end{prop}

\begin{proof}
$\ss{\mathcal{C}_{X}}$ is generated by $\ssig{\mathcal{C}_{X}}=\set{\ns{E}|E\in\mathcal{C}_{X}}$.
It suffices to verify that $\ssig{\mathcal{C}_{X}}$ contains the
diagonal set $\Delta_{\ns{X}}$ of $\ns{X}\times\ns{X}$ and is closed
under finite unions, compositions and inversions. However, it is immediate
from the transfer principle.
\end{proof}
\begin{notation}
We denote the coarse space $\left(\ns{X},\ss{\mathcal{C}_{X}}\right)$
by $S^{c}X$.
\end{notation}

\begin{lem}
\label{lem:S-controlled-sets}Let $X$ be a standard coarse space.
For every subset $E$ of $X\times X$, $E\in\mathcal{C}_{X}$ if and
only if $\ns{E}\in\ss{\mathcal{C}_{X}}$.
\end{lem}

\begin{proof}
The ``only if'' part is trivial. Suppose $\ns{E}\in\ss{\mathcal{C}_{X}}$.
There exists an $F\in\mathcal{C}_{X}$ such that $\ns{E}\subseteq\ns{F}$.
By transfer, $E\subseteq F$, so $E\in\mathcal{C}_{X}$.
\end{proof}
\begin{prop}
\label{prop:rel-btw-S-bornology-and-S-coarse-structure}Let $\mathcal{C}_{X}$
be a coarse structure on a standard set $X$ together with the induced
prebornology $\mathcal{B}_{X}$. The induced prebornology of $\ss{\mathcal{C}_{X}}\restriction\FIN\left(X\right)$
is precisely $\ss{\mathcal{B}_{X}}$.
\end{prop}

\begin{proof}
Suppose that $A\in\ss{\mathcal{B}_{X}}$. There exists a $B\in\mathcal{B}_{X}$
such that $A\subseteq\ns{B}$. Since $B\times B\in\mathcal{C}_{X}$,
$\ns{B}\times\ns{B}=\ns{\left(B\times B\right)}\in\ss{\mathcal{C}_{X}}$.
Therefore $\ns{B}$ is bounded with respect to $\ss{\mathcal{C}_{X}}\restriction\FIN\left(X\right)$,
and so is $A$.

Conversely, suppose that $A$ is bounded with respect to $\ss{\mathcal{C}_{X}}\restriction\FIN\left(X\right)$,
i.e., there exists a bounded set $B$ with respect to $\ss{\mathcal{C}_{X}}$
such that $A=B\cap\FIN\left(X\right)$. By the definition of the induced
prebornology, $B\times B\in\ss{\mathcal{C}_{X}}$ holds. Take an $E\in\mathcal{C}_{X}$
so that $B\times B\subseteq\ns{E}$. If $A$ is empty, then it is
obviously bounded with respect to $\ss{\mathcal{B}_{X}}$. If not,
fix $x_{0}\in A$. Since $x_{0}\in A\subseteq\FIN\left(X\right)$,
there exists an $x_{1}\in X$ such that $x_{0}\sim_{X}x_{1}$. Take
an $F\in\mathcal{C}_{X}$ so that $\left(x_{1},x_{0}\right)\in\ns{F}$.
Then $A\subseteq B\subseteq\ns{E}\left[x_{0}\right]\subseteq\ns{\left(E\circ F\right)}\left[x_{1}\right]\in\ns{\mathcal{B}_{X}}$
by transfer, so $A\in\ss{\mathcal{B}_{X}}$.
\end{proof}
For each map $f\colon X\to Y$ between standard coarse spaces, its
nonstandard extension $\ns{f}\colon\ns{X}\to\ns{Y}$ can naturally
be considered as a map $S^{c}X\to S^{c}Y$. This construction gives
a functor from the category of standard coarse spaces to the category
of (external) coarse spaces. Moreover, this construction not only
preserves but also reflects various properties of the map $f$.
\begin{thm}
\label{thm:S-coarse-str-is-functorial}Let $f\colon X\to Y$ be a
map between standard coarse spaces.
\begin{enumerate}
\item \label{enu:S-coarse-str-functor-1}$f\colon X\to Y$ is bornologous
$\iff$ $\ns{f}\colon S^{c}X\to S^{c}Y$ is bornologous;
\item $f\colon X\to Y$ is effectively proper $\iff$ $\ns{f}\colon S^{c}X\to S^{c}Y$
is effectively proper;
\item $f\colon X\to Y$ is coarsely surjective $\iff$ $\ns{f}\colon S^{c}X\to S^{c}Y$
is coarsely surjective.
\end{enumerate}
\end{thm}

\begin{proof}
\begin{enumerate}
\item Suppose $f\colon X\to Y$ is bornologous. Let $E\in\ss{\mathcal{C}_{X}}$.
Choose an $F\in\mathcal{C}_{X}$ such that $E\subseteq\ns{F}$. Obviously,
$\left(\ns{f}\times\ns{f}\right)\left(E\right)=\ns{\left(f\times f\right)}\left(E\right)\subseteq\ns{\left(f\times f\right)}\left(\ns{F}\right)=\ns{\left(\left(f\times f\right)\left(F\right)\right)}$.
Since $f$ is supposed to be bornologous, $\left(f\times f\right)\left(F\right)\in\mathcal{C}_{Y}$
holds. Hence $\left(\ns{f}\times\ns{f}\right)\left(E\right)\in\ss{\mathcal{C}_{Y}}$.
Conversely, suppose $\ns{f}\colon S^{c}X\to S^{c}Y$ is bornologous.
Let $E\in\mathcal{C}_{X}$. Then $\ns{E}\in\ss{\mathcal{C}_{X}}$,
so $\ns{\left(\left(f\times f\right)\left(E\right)\right)}=\left(\ns{f}\times\ns{f}\right)\left(\ns{E}\right)\in\ss{\mathcal{C}_{Y}}$.
By \prettyref{lem:S-controlled-sets}, we have that $\left(f\times f\right)\left(E\right)\in\mathcal{C}_{Y}$.
\item Similar to \eqref{enu:S-coarse-str-functor-1}.
\item Suppose $f\colon X\to Y$ is coarsely surjective, i.e., there is an
$E\in\mathcal{C}_{Y}$ such that $E\left[f\left(X\right)\right]=Y$.
By transfer, $\ns{E}\left[\ns{f}\left(\ns{X}\right)\right]=\ns{Y}$
and $\ns{E}\in\ss{\mathcal{C}_{Y}}$. Conversely, suppose $\ns{f}\colon S^{c}X\to S^{c}Y$
is coarsely surjective, i.e., there is an $E\in\ss{\mathcal{C}_{Y}}$
such that $E\left[\ns{f}\left(\ns{X}\right)\right]=\ns{Y}$. Let $F\in\mathcal{C}_{Y}$
be such that $E\subseteq\ns{F}$. Then $\ns{F}\left[\ns{f}\left(\ns{X}\right)\right]=\ns{Y}$.
By transfer, $F\left[f\left(X\right)\right]=Y$.\qedhere
\end{enumerate}
\end{proof}
Similarly to \prettyref{prop:inclusion-of-S-born}, the inclusion
map $j_{X}\colon X\hookrightarrow S^{c}X$ can be considered as a
natural embedding.
\begin{prop}
\label{prop:inclusion-of-S-coarse}For each standard coarse space
$X$, the inclusion map $j_{X}\colon X\hookrightarrow S^{c}X$ is
an asymorphic embedding.
\end{prop}

\begin{proof}
Let $E\in\mathcal{C}_{X}$. Since $\left(j_{X}\times j_{X}\right)\left(E\right)=E\subseteq\ns{E}\in\ss{\mathcal{C}_{X}}$,
we have that $\left(j_{X}\times j_{X}\right)\left(E\right)\in\ss{\mathcal{C}_{X}}$.
Therefore $j_{X}$ is bornologous. Next, let $F\in\ss{\mathcal{C}_{X}}$.
There exists an $G\in\mathcal{C}_{X}$ such that $F\subseteq\ns{G}$.
Then, $\left(j_{X}^{-1}\times j_{X}^{-1}\right)\left(F\right)\subseteq\left(j_{X}^{-1}\times j_{X}^{-1}\right)\left(\ns{G}\right)=G\in\mathcal{C}_{X}$
(by transfer). Hence $\left(j_{X}^{-1}\times j_{X}^{-1}\right)\left(F\right)\in\mathcal{C}_{X}$.
Therefore $j_{X}$ is effectively proper. By \prettyref{prop:effectively-properness-and-inverse},
$j_{X}$ is an asymorphic embedding.
\end{proof}

\subsection{S-boundary of coarse spaces}

Giving a bornology on a standard set $X$ is just the same thing as
specifying the infinite points $\INF\left(X\right)$ in $\ns{X}$
in the following sense.
\begin{fact}[{\citep[Corollary 2.12]{Ima19}}]
Let $X$ be a set. For each bornology on $X$, the infinite part
$\INF\left(X\right)$ is a monadic subset of $\ns{X}$ disjoint with
$X$. Conversely, for each monadic subset $I$ of $\ns{X}$, if $I$
is disjoint with $X$, then there is a unique bornology on $X$ such
that $I=\INF\left(X\right)$.
\end{fact}

The set $\INF\left(X\right)$ is a \emph{tabula rasa}, i.e., has no
structure. On the other hand, if $X$ is a coarse space, $\INF\left(X\right)$
is equipped with an additional structure, the subspace coarse structure
induced from $S^{c}X$. We call this subspace the S-boundary of $X$.
\begin{defn}
The \emph{S-boundary} $\partial_{S}X$ of a standard coarse space
$X$ is the subspace $\INF\left(X\right)$ of $S^{c}X$.
\end{defn}

We first consider two examples of S-boundaries which have the same
underlying set but different coarse structures.
\begin{example}
\label{exa:two-coarse-structures-of-R}Consider the real line $\mathbb{R}$
endowed with the usual bornology. There are two distinct galactic
equivalence relations on $\ns{\mathbb{R}}$:
\begin{align*}
x\sim_{\mathbb{R}}y & \iff x-y\in\FIN\left(\mathbb{R}\right),\\
x\sim_{\mathbb{R}}'y & \iff x=y\text{ or }x,y\in\FIN\left(\mathbb{R}\right)
\end{align*}
that correspond to two different coarse structures $\mathcal{C}_{\mathbb{R}}$
and $\mathcal{C}_{\mathbb{R}}'$ on $\mathbb{R}$ (see \citep[Theorem 3.6]{Ima19}).
We denote $\left(\mathbb{R},\mathcal{C}_{\mathbb{R}}\right)$ and
$\left(\mathbb{R},\mathcal{C}_{\mathbb{R}}'\right)$ by $\mathbb{R}$
and $\mathbb{R}'$, respectively. (Recall that galactic equivalence
relations one-to-one correspond to coarse structures .) Clearly the
underlying sets of $\partial_{S}\mathbb{R}$ and $\partial_{S}\mathbb{R}'$
are identical. The connected components of $\partial_{S}\mathbb{R}$
are of the form $x+\FIN\left(\mathbb{R}\right)$, where $x\in\INF\left(\mathbb{R}\right)$.
On the other hand, the connected components of $\partial_{S}\mathbb{R}'$
are singletons $\set{x}$, where $x\in\INF\left(\mathbb{R}\right)$.
So $\partial_{S}\mathbb{R}$ and $\partial_{S}\mathbb{R}'$ are different
as coarse spaces. See also \prettyref{cor:thin-and-S-boundary} and
\prettyref{thm:thin-and-satellite}.
\end{example}

This example suggests that if $X$ and $Y$ are distinct coarse spaces
with the same underlying set, then $\partial_{S}X$ and $\partial_{S}Y$
are different. This statement is true as we shall now prove.
\begin{thm}
\label{thm:functor-property-of-S-boundaries}Let $X$ and $Y$ be
standard coarse spaces and let $f\colon X\to Y$ be a proper bornological
map. Then $f\colon X\to Y$ is bornologous (resp. effectively proper)
if and only if $\ns{f}\colon\partial_{S}X\to\partial_{S}Y$ is bornologous
(resp. effectively proper).
\end{thm}

\begin{proof}
The well-definedness of $\ns{f}\colon\partial_{S}X\to\partial_{S}Y$
follows from the nonstandard characterisation of properness \citep[Theorem 2.28]{Ima19}.
Suppose $f$ is bornologous (resp. effectively proper). Then $\ns{f}\colon S^{c}X\to S^{c}Y$
is bornologous (resp. effectively proper) by \prettyref{thm:S-coarse-str-is-functorial}.
Hence $\ns{f}\colon\partial_{S}X\to\partial_{S}Y$ is bornologous
(resp. effectively proper).

Suppose $\ns{f}\colon\partial_{S}X\to\partial_{S}Y$ is bornologous.
Let $x,y\in\ns{X}$ and assume that $x\sim_{X}y$. Case I: $x,y\in\FIN\left(X\right)$.
Pick a $z\in X$ so that $y\sim_{X}x\sim_{X}z$. Since $f$ is bornological
at $z$, we have that $\ns{f}\left(x\right)\sim_{Y}f\left(z\right)\sim_{Y}\ns{f}\left(y\right)$
by \citep[Theorem 2.24]{Ima19}. Case II: $x,y\in\INF\left(X\right)$.
There is an $E\in\mathcal{C}_{X}$ so that $\left(x,y\right)\in\ns{E}$.
Since $\ns{E}\restriction\INF\left(X\right)\in\ss{\mathcal{C}_{X}}\restriction\INF\left(X\right)$,
we have $\left(\ns{f}\times\ns{f}\right)\left(\ns{E}\restriction\INF\left(X\right)\right)\in\ss{\mathcal{C}_{Y}}\restriction\INF\left(Y\right)$.
Let $F\in\mathcal{C}_{Y}$ be such that $\left(\ns{f}\times\ns{f}\right)\left(\ns{E}\restriction\INF\left(X\right)\right)\subseteq\ns{F}\restriction\INF\left(Y\right)$.
Then $\left(\ns{f}\left(x\right),\ns{f}\left(y\right)\right)\in\ns{F}$,
and therefore $\ns{f}\left(x\right)\sim_{Y}\ns{f}\left(y\right)$.
The other cases are impossible. By \citep[Theorem 3.23]{Ima19}, $f$
is bornologous.

Suppose $\ns{f}\colon\partial_{S}X\to\partial_{S}Y$ is effectively
proper. Let $x,y\in\ns{X}$ and assume that $\ns{f}\left(x\right)\sim_{Y}\ns{f}\left(y\right)$.
Case I: $\ns{f}\left(x\right),\ns{f}\left(y\right)\in\FIN\left(Y\right)$.
Pick a $z\in Y$ so that $\ns{f}\left(x\right)\sim_{Y}z\sim_{Y}\ns{f}\left(y\right)$.
There exists a bounded subset $B$ of $Y$ such that $\ns{f}\left(x\right),z,\ns{f}\left(y\right)\in\ns{B}$,
so $x,y\in\ns{f}^{-1}\left(\ns{B}\right)$. Since $f$ is proper,
$f^{-1}\left(B\right)$ is bounded in $X$. Hence $x\sim_{X}y$. Case
II: $\ns{f}\left(x\right),\ns{f}\left(y\right)\in\INF\left(Y\right)$.
Let $E\in\mathcal{C}_{Y}$ be such that $\left(\ns{f}\left(x\right),\ns{f}\left(y\right)\right)\in\ns{E}$.
Since $\ns{E}\restriction\INF\left(Y\right)\in\ss{\mathcal{C}_{Y}}\restriction\INF\left(Y\right)$,
we have $\left(\ns{f}^{-1}\times\ns{f}^{-1}\right)\left(\ns{E}\restriction\INF\left(Y\right)\right)\in\ss{\mathcal{C}_{X}}\restriction\INF\left(X\right)$,
i.e., there is an $F\in\mathcal{C}_{X}$ such that $\left(\ns{f}^{-1}\times\ns{f}^{-1}\right)\left(\ns{E}\restriction\INF\left(Y\right)\right)\subseteq\ns{F}\restriction\INF\left(X\right)$.
Then $\left(x,y\right)\in\ns{F}$, and therefore $x\sim_{X}y$. The
other cases are impossible. Hence $f$ is effectively proper by \prettyref{thm:nonst-charact-eff-proper}.
\end{proof}
\begin{cor}
Let $X$ and $X'$ be standard coarse spaces. Then $X=X'$ if and
only if $\mathcal{B}_{X}=\mathcal{B}_{X'}$ and $\partial_{S}X=\partial_{S}X'$,
where $\mathcal{B}_{X}$ and $\mathcal{B}_{X'}$ are the induced prebornologies
of $X$ and $X'$, respectively.
\end{cor}

\begin{proof}
Apply \prettyref{thm:functor-property-of-S-boundaries} to the identity
map $\id_{X}$.
\end{proof}
\begin{cor}
\label{cor:S-corona-determines-coarse-structure}Let $X$ and $X'$
be standard connected coarse spaces with the same underlying set.
Then $X=X'$ if and only if $\partial_{S}X=\partial_{S}X'$.
\end{cor}

This means that a coarse structure is determined by the structure
of ``the space at infinity''. This phenomenon is ubiquitous. For example,
Dydak \citep{Dyd19} introduced the notion of simple ends and simple
coarse structures for prebornological spaces. A \emph{simple end}
in a prebornological space $X$ is a proper map $\mathbb{N}\to X$,
or in other words, a divergent sequence in $X$. A \emph{simple coarse
structure} on $X$ is an equivalence relation $\mathcal{SCS}_{X}$
on the set of all simple ends in $X$. Intuitively, each simple end
represents an ideal infinite point; and two simple ends represent
the same infinite point if and only if it is $\mathcal{SCS}_{X}$-equivalent.
Each simple coarse structure $\mathcal{SCS}_{X}$ on $X$ induces
a coarse structure $\mathbb{CS}\left(\mathcal{SCS}_{X}\right)$ on
$X$. Conversely, each coarse structure $\mathcal{C}_{X}$ on $X$
induces a simple coarse structure $\mathbb{SCS}\left(\mathcal{C}_{X}\right)$
on $X$. Those two constructions $\mathbb{CS}$ and $\mathbb{SCS}$
are inverses to each other for some cases (but not in general). See
\citep{Dyd19} for more details. The notion of topological ends goes
back to Freudenthal \citep{Fre31}. It is one conception of ``the
space at infinity'' of a topological space. The nonstandard treatment
of topological ends can be found in Goldbring \citep{Gol11} and Insall
\emph{et al.} \citep{ILM14}. Some conceptions of ``the space at infinity''
of a coarse space are studied in, e.g., Hartmann \citep{Har19} and
Grzegrzolka and Siegert \citep{GS19}. The following is an analogous
result to \citep[Lemma 36]{Har19}.
\begin{thm}
Let $f\colon X\to Y$ be a proper map between standard coarse spaces,
where $X$ is non-empty and $Y$ is connected. If $\ns{f}\colon\partial_{S}X\to\partial_{S}Y$
is coarsely surjective, then $f$ is coarsely surjective.
\end{thm}

\begin{proof}
Fix an $x_{0}\in X$, then $\FIN\left(Y\right)=G_{Y}^{c}\left(f\left(x_{0}\right)\right)\subseteq G_{Y}^{c}\left(\ns{f}\left(\FIN\left(X\right)\right)\right)$
by \citep[Corollary 3.13]{Ima19}. Take an $E\in\mathcal{C}_{Y}$
so that $\ns{E}\left[\ns{f}\left(\INF\left(X\right)\right)\right]\supseteq\INF\left(Y\right)$,
then $G_{X}^{c}\left(\ns{f}\left(\INF\left(X\right)\right)\right)\supseteq\ns{E}\left[\ns{f}\left(\INF\left(X\right)\right)\right]\supseteq\INF\left(Y\right)$.
Hence
\begin{align*}
\ns{Y} & =\FIN\left(Y\right)\cup\INF\left(Y\right)\\
 & \subseteq G_{X}^{c}\left(\ns{f}\left(\FIN\left(X\right)\right)\right)\cup G_{X}^{c}\left(\ns{f}\left(\INF\left(X\right)\right)\right)\\
 & =G_{X}^{c}\left(\ns{f}\left(\FIN\left(X\right)\right)\cup\ns{f}\left(\INF\left(X\right)\right)\right)\\
 & =G_{X}^{c}\left(\ns{f}\left(\FIN\left(X\right)\cup\INF\left(X\right)\right)\right)\\
 & =G_{X}^{c}\left(\ns{f}\left(\ns{X}\right)\right).
\end{align*}
By \prettyref{thm:nonst-charact-coarse-surjection}, $f$ is coarsely
surjective.
\end{proof}
\begin{thm}
\label{thm:S-corona-preserves-coarse-surjectivity}Let $f\colon X\to Y$
be a proper bornological map between standard coarse spaces. If $f\colon X\to Y$
is coarsely surjective, then $\ns{f}\colon\partial_{S}X\to\partial_{S}Y$
is coarsely surjective.
\end{thm}

\begin{proof}
Let $E\in\mathcal{C}_{Y}$ be such that $E\left[f\left(X\right)\right]=Y$.
By transfer, $\ns{E}\left[\ns{f}\left(\ns{X}\right)\right]=\ns{Y}\supseteq\INF\left(Y\right)$.
Let $y\in\INF\left(Y\right)$. Take an $x\in\ns{X}$ such that $y\in\ns{E}\left[\ns{f}\left(x\right)\right]$.
Since $\ns{f}\left(x\right)\sim_{Y}y\in\INF\left(Y\right)$, we have
that $\ns{f}\left(x\right)\in\INF\left(Y\right)$, so $x\in\INF\left(X\right)$
by \citep[Theorem 2.28]{Ima19}. Hence $\INF\left(Y\right)\subseteq\left(\ns{E}\restriction\INF\left(Y\right)\right)\left[\ns{f}\left(\INF\left(X\right)\right)\right]$,
where $\ns{E}\restriction\INF\left(Y\right)\in\ss{\mathcal{C}_{Y}}\restriction\INF\left(Y\right)$.
Therefore $\ns{f}\colon\partial_{S}X\to\partial_{S}Y$ is coarsely
surjective.
\end{proof}
\begin{cor}
\label{cor:S-corona-preserves-coarse-equivalences}Let $f\colon X\to Y$
be a coarse equivalence between standard coarse spaces. Then $\ns{f}\colon\partial_{S}X\to\partial_{S}Y$
is a coarse equivalence.
\end{cor}

\begin{proof}
By \prettyref{prop:coarse-surjectivity-and-bornotopy-inverse}, $f$
is effectively proper, bornologous and coarsely surjective. By \prettyref{thm:functor-property-of-S-boundaries}
and \prettyref{thm:S-corona-preserves-coarse-surjectivity}, $\ns{f}$
is effectively proper, bornologous and coarsely surjective. Again
by \prettyref{prop:coarse-surjectivity-and-bornotopy-inverse}, $\ns{f}$
is a coarse equivalence.
\end{proof}

\section{\label{sec:Large-scale-structures-on-Pow}Size properties and coarse
hyperspaces}

Several concepts of combinatorial size for subsets of a group have
been developed and well-studied (e.g. \citep{BM99,BM01,DMM03,Pro04}).
For example, a subset $L$ of a group $\varGamma$ is said to be\emph{
left large} (resp. \emph{right large}) if $K\cdot L=\varGamma$ (resp.
$L\cdot K=\varGamma$) for some finite subset $K$ of $\varGamma$.
Such properties can be extended to general coarse spaces \citep{PB03,PZ07}.
In this section, we study size properties of subsets of coarse spaces.

\subsection{Size of subsets of coarse spaces}

We first consider the following size properties.
\begin{defn}[Standard; \citep{PB03,DZ17}]
Let $X$ be a coarse space. A subset $A$ of $X$ is said to be
\begin{enumerate}
\item \emph{large} (a.k.a. \emph{coarsely dense}) if $E\left[A\right]=X$
for some $E\in\mathcal{C}_{X}$;
\item \emph{slim} if $E\left[A\right]\neq X$ for all $E\in\mathcal{C}_{X}$;
\item \emph{thick} if $\intr_{X,E}A\neq\varnothing$ for all $E\in\mathcal{C}_{X}$;
\item \emph{meshy} if $\intr_{X,E}A=\varnothing$ for some $E\in\mathcal{C}_{X}$.
\end{enumerate}
Let $\mathcal{L}\left(X\right)$ and $\mathcal{M}\left(X\right)$
be the family of large and meshy subsets of $X$, respectively.

\end{defn}

\begin{example}
Let $\varGamma$ be a group. Then the finite bornology $\mathcal{P}_{f}\left(\varGamma\right)$
on $\varGamma$ induces two coarse structures $\mathcal{C}_{\varGamma,l}$
and $\mathcal{C}_{\varGamma,r}$ on $\varGamma$, the left coarse
structure and the right coarse structure, whose finite closeness relations
are given by $x\sim_{\varGamma,l}y\iff x^{-1}y\in\varGamma$ and $x\sim_{\varGamma,r}y\iff xy^{-1}\in\varGamma$
for $x,y\in\ns{\varGamma}$, respectively (see also \citep[Example 3.18]{Ima19}).
It is easy to see that a subset $L$ of $\varGamma$ is left large
(resp. right large) if and only if $L$ is large with respect to $\mathcal{C}_{\varGamma,r}$
(resp. $\mathcal{C}_{\varGamma,l}$). In other words, \emph{left}
largeness is largeness with respect to the \emph{right} coarse structure;
and \emph{right} largeness is largeness with respect to the \emph{left}
coarse structure.
\end{example}

\begin{thm}
\label{thm:nonst-characterisation-of-size-I}Let $X$ be a standard
coarse space and $A$ a subset of $X$.
\begin{enumerate}
\item \label{enu:large-is-dense}$A$ is large $\iff$ $G_{X}^{c}\left(\ns{A}\right)=\ns{X}$
($\iff$ $\ns{A}$ is $G_{X}^{c}$-dense);
\item $A$ is slim $\iff$ $G_{X}^{c}\left(\ns{A}\right)\neq\ns{X}$ ($\iff$
$\ns{A}$ is not $G_{X}^{c}$-dense);
\item \label{enu:thick-is-open}$A$ is thick $\iff$ $C_{X}^{c}\left(\ns{A}\right)\neq\varnothing$
($\iff$ $\ns{A}$ has non-empty $C_{X}^{c}$-interior);
\item $A$ is meshy $\iff$ $C_{X}^{c}\left(\ns{A}\right)=\varnothing$
($\iff$ $\ns{A}$ has empty $C_{X}^{c}$-interior).
\end{enumerate}
\end{thm}

\begin{proof}
\begin{enumerate}
\item Suppose $A$ is large, i.e., there is an $E\in\mathcal{C}_{X}$ so
that $E\left[A\right]=X$. By transfer, $\ns{X}=\ns{E}\left[\ns{A}\right]\subseteq G_{X}^{c}\left(\ns{A}\right)\subseteq\ns{X}$.
Hence $G_{X}^{c}\left(\ns{A}\right)=\ns{X}$. Conversely, suppose
$G_{X}^{c}\left(\ns{A}\right)=\ns{X}$. By \prettyref{lem:General-Approximation-Lemma},
there exists an $E\in\ns{\mathcal{C}_{X}}$ such that ${\sim_{X}}\subseteq E$.
Then $\ns{X}=G_{X}^{c}\left(\ns{A}\right)\subseteq E\left[\ns{A}\right]\subseteq\ns{X}$,
so $E\left[\ns{A}\right]=\ns{X}$. By transfer, there exists an $F\in\mathcal{C}_{X}$
such that $F\left[A\right]=X$.
\item Immediate from \eqref{enu:large-is-dense}.
\item Suppose $A$ is thick, i.e., $\intr_{X,E}A\neq\varnothing$ for all
$E\in\mathcal{C}_{X}$. By transfer, $\intr_{\ns{X},E}\ns{A}\neq\varnothing$
for all $E\in\ns{\mathcal{C}_{X}}$. Choose an $F\in\ns{\mathcal{C}_{X}}$
so that ${\sim_{X}}\subseteq F$ by \prettyref{lem:General-Approximation-Lemma}.
Then $\intr_{\ns{X},F}\ns{A}\neq\varnothing$, i.e., $F\left[x\right]\subseteq\ns{A}$
for some $x\in\ns{A}$. Since ${\sim_{X}}\subseteq F$, we have that
$G_{X}^{c}\left(x\right)\subseteq F\left[x\right]\subseteq\ns{A}$.
Hence $x\in C_{X}^{c}\left(\ns{A}\right)\neq\varnothing$. Conversely,
suppose $C_{X}^{c}\left(\ns{A}\right)\neq\varnothing$. Fix an $x\in C_{X}^{c}\left(\ns{A}\right)$.
Let $E\in\mathcal{C}_{X}$. Then $\ns{E}\left[x\right]\subseteq G_{X}^{c}\left(x\right)\subseteq\ns{A}$,
so $x\in\intr_{\ns{X},\ns{E}}\ns{A}\neq\varnothing$. By transfer,
$\intr_{X,E}A\neq\varnothing$.
\item Immediate from \eqref{enu:thick-is-open}.\qedhere
\end{enumerate}
\end{proof}
As Protasov and Zarichnyi \citep[p. 172]{PZ07} pointed out, large
and thick subsets of a coarse space can be considered as large-scale
counterparts of dense and open subsets of a topological space, respectively.
Indeed, many results on size properties can be proved \emph{analogically}
to their small-scale (topological) counterparts. On the other hand,
\prettyref{thm:nonst-characterisation-of-size-I} indicates that large
and thick subsets \emph{precisely} correspond to dense and ``with
non-empty interior'' subsets, respectively. Hence, using our nonstandard
characterisations, we can deduce many large-scale results from their
small-scale counterparts not only \emph{analogically} but also \emph{logically}.
\begin{cor}[Standard]
Let $X$ be a coarse space and $A\subseteq B\subseteq X$. If $A$
is large in $B$ and $B$ is large in $X$, then $A$ is large in
$X$.
\end{cor}

\begin{proof}
By \prettyref{thm:nonst-characterisation-of-size-I}, $\ns{B}$ is
a $G_{X}^{c}$-dense subset of $\ns{X}$, and $\ns{A}$ is a $G_{B}^{c}$-dense
subset of $\ns{B}$, so $\ns{A}$ is a $G_{X}^{c}$-dense subset of
$\ns{X}$. (Note that $G_{B}^{c}=G_{X}^{c}\cap\ns{B}$ holds by \citep[Example 3.19]{Ima19},
i.e., $\ns{B}$ is a topological subspace of $\ns{X}$.) Hence $A$
is large in $X$ by \prettyref{thm:nonst-characterisation-of-size-I}.
\end{proof}
\begin{cor}[{Standard; \citep[Proposition 9.1.2]{PZ07}}]
\label{cor:intersection-of-thick-and-large}Let $X$ be a coarse
space and $A$ a subset of $X$. The following are equivalent:
\begin{enumerate}
\item $A$ is thick;
\item $L\cap A\neq\varnothing$ for each large subset $L$ of $X$.
\end{enumerate}
\end{cor}

\begin{proof}
Suppose $A$ is thick. Let $L$ be a large subset of $X$. By \prettyref{thm:nonst-characterisation-of-size-I},
$\ns{A}$ has non-empty interior and $\ns{L}$ is dense. It is evident
that the intersection of a dense subset and a set with non-empty interior
is non-empty. Hence $\ns{L}\cap\ns{A}\neq\varnothing$. We have $L\cap A\neq\varnothing$
by transfer.

Conversely, suppose $A$ is not thick. Set $L=X\setminus A$. Clearly
$L\cap A=\varnothing$. By \prettyref{thm:nonst-characterisation-of-size-I},
$\ns{A}$ has empty interior. It is also evident that the complement
of a subset with empty-interior is dense. Hence $\ns{L}=\ns{X}\setminus\ns{A}$
is dense. By \prettyref{thm:nonst-characterisation-of-size-I}, $L$
is large.
\end{proof}
\begin{cor}[Standard]
\label{cor:bounded-implies-meshy}Let $X$ be an unbounded connected
coarse space. Then $\mathcal{B}_{X}\subseteq\mathcal{M}\left(X\right)$.
\end{cor}

\begin{proof}
Let $B\in\mathcal{B}_{X}$. Since $X$ is unbounded, $X\setminus B$
is non-empty. Fix an $x_{0}\in X\setminus B$. Let $x\in\ns{X}$.
Case I: $x\in\FIN\left(X\right)$. By \citep[Proposition 2.11]{Ima19},
$x_{0}\in\FIN\left(X\right)=G_{X}^{c}\left(x\right)$ but $x_{0}\notin\ns{B}$,
so $G_{X}^{c}\left(x\right)\nsubseteq\ns{B}$. Case II: $x\in\INF\left(X\right)$.
By \citep[Proposition 2.6]{Ima19}, $x\in G_{X}^{c}\left(x\right)$
but $x\notin\FIN\left(X\right)\supseteq\ns{B}$, so $G_{X}^{c}\left(x\right)\nsubseteq\ns{B}$.
In both cases, we have that $x\notin C_{X}^{c}\left(\ns{B}\right)$.
Hence $C_{X}^{c}\left(\ns{B}\right)=\varnothing$. By \prettyref{thm:nonst-characterisation-of-size-I},
$B$ is meshy.
\end{proof}
We next consider four more complicated size properties.
\begin{defn}[Standard; \citep{PB03,DZ17}]
Let $X$ be a coarse space. A subset $A$ of $X$ is said to be
\begin{enumerate}
\item \emph{piecewise large} if $E\left[A\right]$ is thick for some $E\in\mathcal{C}_{X}$;
\item \emph{small} if $X\setminus E\left[A\right]$ is large for all $E\in\mathcal{C}_{X}$;
\item \emph{extralarge} if $\intr_{X,E}A$ is large for all $E\in\mathcal{C}_{X}$;
\item \emph{with slim interior} if $\intr_{X,E}A$ is slim for some $E\in\mathcal{C}_{X}$.
\end{enumerate}
\end{defn}

\begin{thm}
\label{thm:nonst-characterisation-of-size-II}Let $X$ be a standard
coarse space and $A$ a subset of $X$.
\begin{enumerate}
\item $A$ is piecewise large $\iff$ $C_{X}^{c}\left(\ns{E}\left[\ns{A}\right]\right)\neq\varnothing$
for some $E\in\mathcal{C}_{X}$;
\item \label{enu:small-is-nowhere-dense}$A$ is small $\iff$ $C_{X}^{c}\left(\ns{E}\left[\ns{A}\right]\right)=\varnothing$
for all $E\in\mathcal{C}_{X}$;
\item $A$ is extralarge $\iff$ $G_{X}^{c}\left(\intr_{\ns{X},\ns{E}}\ns{A}\right)=\ns{X}$
for all $E\in\mathcal{C}_{X}$;
\item $A$ is with slim interior $\iff$ $G_{X}^{c}\left(\intr_{\ns{X},\ns{E}}\ns{A}\right)\neq\ns{X}$
for some $E\in\mathcal{C}_{X}$.
\end{enumerate}
\end{thm}

\begin{proof}
Immediate from \prettyref{thm:nonst-characterisation-of-size-I}.
\end{proof}
The definitions of piecewise large, small, extralarge, and with slim
interior subsets might look slightly complicated, as compared with
those of large, slim, thick and meshy subsets. However, there are
simpler (lattice-theoretic) characterisations of smallness and extralargeness.
We provide a nonstandard proof of the characterisation of extralargeness.
\begin{thm}[{Standard; \citep[Theorem 11.1]{PB03}}]
\label{thm:characterisation-of-extralargeness}Let $X$ be a standard
coarse space and $A$ a subset of $X$. The following are equivalent:
\begin{enumerate}
\item $A$ is extralarge;
\item $L\cap A$ is large for each large subset $L$ of $X$.
\end{enumerate}
\end{thm}

\begin{proof}
Suppose $A$ is extralarge. Let $L$ be a large subset of $X$. There
is an $E\in\mathcal{C}_{X}$ so that $E^{-1}\left[L\right]=X$, i.e.,
$L\cap E\left[x\right]\neq\varnothing$ for all $x\in X$. Since $A$
is extralarge, $G_{X}^{c}\left(\intr_{\ns{X},\ns{E}}\ns{A}\right)=\ns{X}$
holds by \prettyref{thm:nonst-characterisation-of-size-II}. Let $x\in\ns{X}$.
Then there exists a $y\in\intr_{\ns{X},\ns{E}}\ns{A}$ such that $x\sim_{X}y$.
By transfer, we have that $\ns{L}\cap\ns{E}\left[y\right]\neq\varnothing$.
Choose a $z\in\ns{L}\cap\ns{E}\left[y\right]$. Then $x\sim_{X}y\sim_{X}z\in\ns{L}\cap\ns{E}\left[y\right]\subseteq\ns{L}\cap\ns{A}$.
Hence $x\in G_{X}^{c}\left(\ns{L}\cap\ns{A}\right)$. By \prettyref{thm:nonst-characterisation-of-size-I},
$L\cap A$ is large.

Suppose $A$ is not extralarge, i.e., $\intr_{X,E}A$ is not large
for some $E\in\mathcal{C}_{X}$. Let $L=\intr_{X,E}A\cup\bigcup_{x\in X\setminus\intr_{X,E}A}\left(E\left[x\right]\setminus A\right)$.
For each $\ns{x}\in\ns{X}\setminus\intr_{\ns{X},\ns{E}}\ns{A}$, since
$\ns{E}\left[x\right]\setminus\ns{A}$ is non-empty by transfer, it
follows that $x\in G_{X}^{c}\left(\ns{L}\right)$. Hence $G_{X}^{c}\left(\ns{L}\right)=\ns{X}$.
By \prettyref{thm:nonst-characterisation-of-size-I}, $L$ is large.
On the other hand, $L\cap A=\intr_{X,E}A$ is not large.
\end{proof}
\begin{cor}[{Standard; \citep[Theorem 11.1]{PB03}}]
Let $X$ be a standard coarse space and $A$ a subset of $X$. The
following are equivalent:
\begin{enumerate}
\item $A$ is small;
\item $L\setminus A$ is large for each large subset $L$ of $X$.
\end{enumerate}
\end{cor}

\begin{proof}[Proof (Standard)]
Apply \prettyref{thm:characterisation-of-extralargeness} to the
complement $X\setminus A$.
\end{proof}
\begin{cor}[{Standard; \citep[Proposition 2.11]{DZ17}}]
Let $X$ be a coarse space and $A$ a subset of $X$. The following
are equivalent:
\begin{enumerate}
\item $A$ is small;
\item $A\cup B$ is meshy for each meshy subset $B$ of $X$.
\end{enumerate}
\end{cor}

\begin{proof}[Proof (Standard)]
$A$ is small $\iff$ $L\setminus A\in\mathcal{L}\left(X\right)$
for all $L\in\mathcal{L}\left(X\right)$ $\iff$ $\left(X\setminus B\right)\setminus A=X\setminus\left(A\cup B\right)\in\mathcal{L}\left(X\right)$
for all $B\in\mathcal{M}\left(X\right)$ $\iff$ $A\cup B\in\mathcal{M}\left(X\right)$
for all $B\in\mathcal{M}\left(X\right)$.
\end{proof}
Protasov and Zarichnyi \citep[p. 172]{PZ07} pointed out that small
subsets of a coarse space can be considered as the large-scale counterpart
of nowhere dense subsets of a topological space. In the light of the
topology of $\ns{X}$ defined by \prettyref{cor:galactic-topology-of-starX},
small subsets do not precisely correspond to nowhere dense subsets:
if $\ns{A}$ is nowhere dense, then $\ns{A}\subseteq G_{X}^{c}\left(\ns{A}\right)=C_{X}^{c}\left(G_{X}^{c}\left(\ns{A}\right)\right)=\varnothing$
by \prettyref{thm:galaxy-and-core}, so $\ns{A}$ must be empty; however,
every unbounded connected coarse space has a non-empty small subset.
\begin{thm}[{Standard; \citep[Theorem 2.14]{DZ17}}]
\label{thm:unboundedness-and-small-subsets}Let $X$ be a non-empty
connected coarse space. The following are equivalent:
\begin{enumerate}
\item \label{enu:unbounded-and-small-finite-cond-1}$X$ is unbounded;
\item \label{enu:unbounded-and-small-finite-cond-2}every finite subset
of $X$ is small.
\end{enumerate}
\end{thm}

\begin{proof}
Suppose that \eqref{enu:unbounded-and-small-finite-cond-2} does not
hold, i.e., there is a non-small (i.e. piecewise large) finite subset
$A$ of $X$. For some $E\in\mathcal{C}_{X}$, $C_{X}^{c}\left(\ns{E}\left[\ns{A}\right]\right)\neq\varnothing$
holds by \prettyref{thm:nonst-characterisation-of-size-II}. Note
that $E\left[A\right]$ is bounded. For any $x\in C_{X}^{c}\left(\ns{E}\left[\ns{A}\right]\right)$,
$G_{X}^{c}\left(x\right)\subseteq\ns{E}\left[\ns{A}\right]$ holds.
Since $X$ is connected, $\FIN\left(X\right)\subseteq\ns{E}\left[\ns{A}\right]$
by \citep[Proposition 2.11]{Ima19}. Then $X\subseteq\FIN\left(X\right)\subseteq\ns{E}\left[\ns{A}\right]$,
so $X\subseteq E\left[A\right]$ by transfer. (More rigorously, for
each $x\in X$, apply the transfer principle to ``$x\in\ns{E}\left[\ns{A}\right]$''
to obtain ``$x\in E\left[A\right]$''.) Hence $X$ is bounded. Thus
the implication \eqref{enu:unbounded-and-small-finite-cond-1} to
\eqref{enu:unbounded-and-small-finite-cond-2} is proved.

Suppose \eqref{enu:unbounded-and-small-finite-cond-1} does not hold.
Since $X$ is bounded, $E=X\times X\in\mathcal{C}_{X}$. Fix an $x\in X$.
Obviously $E\left[x\right]=X$ holds. By transfer, $\ns{E}\left[x\right]=\ns{X}$.
Hence $C_{X}^{c}\left(\ns{E}\left[x\right]\right)=\ns{X}\neq\varnothing$.
By \prettyref{thm:nonst-characterisation-of-size-II}, $\set{x}$
is not small. Thus the implication \eqref{enu:unbounded-and-small-finite-cond-2}
to \eqref{enu:unbounded-and-small-finite-cond-1} is proved.
\end{proof}
With a similar argument, we can prove the following theorem.
\begin{thm}[{Standard; \citep[Theorem 4.16]{DZ17}}]
Let $X$ be a coarse space. The following are equivalent:
\begin{enumerate}
\item \label{enu:bounded-components-and-PL-cond-1}every connected component
of $X$ is bounded;
\item \label{enu:bounded-components-and-PL-cond-2}every non-empty subset
of $X$ is piecewise large.
\end{enumerate}
\end{thm}

\begin{proof}
For $x\in X$, let $Q_{x}$ be the connected component of $x$, namely,
$\bigcup_{E\in\mathcal{C}_{X}}E\left[x\right]$. Observe that $\ns{Q_{x}}=\bigcup_{E\in\ns{\mathcal{C}_{X}}}E\left[x\right]$;
$G_{X}^{c}\left(x\right)=\bigcup_{E\in\mathcal{C}_{X}}\ns{E}\left[x\right]$;
$Q_{x}\subseteq G_{X}^{c}\left(x\right)\subseteq\ns{Q_{x}}$ (see
also \citep[Corollary 3.13]{Ima19}).

\eqref{enu:bounded-components-and-PL-cond-1}$\Rightarrow$\eqref{enu:bounded-components-and-PL-cond-2}:
let $A$ be a non-empty subset of $X$. Fix $x_{0}\in A$. Since the
connected component $Q_{x_{0}}$ is bounded, there exists an $E\in\mathcal{C}_{0}$
such that $Q_{x_{0}}\subseteq E\left[x_{0}\right]$. By transfer,
$G_{X}^{c}\left(x_{0}\right)\subseteq\ns{Q_{x_{0}}}\subseteq\ns{E}\left[x_{0}\right]\subseteq\ns{E}\left[\ns{A}\right]$,
and therefore $x_{0}\in C_{X}^{c}\left(\ns{E}\left[\ns{A}\right]\right)\neq\varnothing$.
By \prettyref{thm:nonst-characterisation-of-size-II}, $A$ is piecewise
large.

\eqref{enu:bounded-components-and-PL-cond-2}$\Rightarrow$\eqref{enu:bounded-components-and-PL-cond-1}:
Let $x_{0}\in X$. Since $\set{x_{0}}$ is piecewise large, $C_{X}^{c}\left(\ns{E}\left[x_{0}\right]\right)\neq\varnothing$
holds for some $E\in\mathcal{C}_{X}$ by \prettyref{thm:nonst-characterisation-of-size-II}.
For $x\in C_{X}^{c}\left(\ns{E}\left[x_{0}\right]\right)$, $G_{X}^{c}\left(x\right)\subseteq\ns{E}\left[x_{0}\right]$,
so $x\sim_{X}x_{0}$ and $G_{X}^{c}\left(x\right)=G_{X}^{c}\left(x_{0}\right)$.
Then $Q_{x_{0}}\subseteq G_{X}^{c}\left(x_{0}\right)\subseteq\ns{E}\left[x_{0}\right]$,
so $Q_{x_{0}}\subseteq E\left[x_{0}\right]$ by transfer, and therefore
$Q_{x_{0}}$ is bounded. (More rigorously, for each $x\in Q_{x_{0}}$,
apply the transfer principle to ``$x\in\ns{E}\left[x_{0}\right]$''.)
\end{proof}

\subsection{Thin subsets and slowly oscillating maps}

We provide a nonstandard characterisation of thin coarse spaces, and
apply it to proving some standard characterisations of thinness (in
terms of slow oscillation and meshiness).
\begin{defn}[Standard; \citep{Pro03,LP10}]
A subset $A$ of a coarse space $X$ is said to be \emph{thin} (a.k.a.
\emph{pseudodiscrete}) if for every $E\in\mathcal{C}_{X}$ there exists
a bounded subset $B$ of $X$ such that $E\left[x\right]\cap E\left[y\right]=\varnothing$
for all distinct $x,y\in A\setminus B$.
\end{defn}

\begin{thm}
\label{thm:nonst-characterisation-of-thin-subsets}Let $A$ be a subset
of a standard connected coarse space $X$. The following are equivalent:
\begin{enumerate}
\item $A$ is thin;
\item $x\sim_{X}y$ implies $x=y$ for all $x,y\in\ns{A}\cap\INF\left(X\right)$;
\item $G_{X}^{c}\left(x\right)\cap G_{X}^{c}\left(y\right)=\varnothing$
for all distinct $x,y\in\ns{A}\cap\INF\left(X\right)$.
\end{enumerate}
\end{thm}

\begin{proof}
Suppose $A$ is thin. Let $x,y\in\ns{A}\cap\INF\left(X\right)$ with
$x\sim_{X}y$. Choose an $E\in\mathcal{C}_{X}$ so that $\left(x,y\right)\in\ns{E}$
and $\left(y,y\right)\in\ns{E}$. Since $A$ is supposed to be thin,
we can find a bounded subset $B$ of $X$ such that $E\left[u\right]\cap E\left[v\right]=\varnothing$
holds for all distinct $u,v\in A\setminus B$. However, $x,y\in\ns{A}\setminus\FIN\left(X\right)\subseteq\ns{A}\setminus\ns{B}$
and $\ns{E}\left[x\right]\cap\ns{E}\left[y\right]\supseteq\set{y}\neq\varnothing$.
By transfer, $x$ and $y$ cannot be distinct, i.e., $x=y$.

Conversely, suppose $A$ is not thin, i.e., there is an $E\in\mathcal{C}_{X}$
such that for any bounded subset $B$ of $X$, $E\left[x\right]\cap E\left[y\right]\neq\varnothing$
holds for some distinct $x,y\in A\setminus B$. By \prettyref{lem:General-Approximation-Lemma},
we can choose a $B\in\ns{\mathcal{B}_{X}}$ so that $\FIN\left(X\right)\subseteq B$.
(Here we used the connectedness of $X$.) By transfer, $\ns{E}\left[x\right]\cap\ns{E}\left[y\right]\neq\varnothing$
for some distinct $x,y\in\ns{A}\setminus B\subseteq\ns{A}\cap\INF\left(X\right)$.
Pick a $z\in\ns{E}\left[x\right]\cap\ns{E}\left[y\right]$, then $x\sim_{X}z\sim_{X}y$.
Hence we have that $x\sim_{X}y$ but $x\neq y$.
\end{proof}
\begin{cor}
\label{cor:thin-and-S-boundary}For every standard connected coarse
space $X$, the following are equivalent:
\begin{enumerate}
\item $X$ is thin;
\item $\partial_{S}X$ is (bornologically) discrete;
\item $\INF\left(X\right)$ is (topologically) $G_{X}^{c}$-discrete.
\end{enumerate}
\end{cor}

This is the reason why `thin' is also called `pseudodiscrete'.
\begin{example}
Consider the set $X=\set{n^{2}|n\in\mathbb{N}}$ endowed with the
usual metric $d_{X}\left(n,m\right)=\left|n-m\right|$. For any distinct
$n^{2},m^{2}\in\INF\left(X\right)$, since $\left|n^{2}-m^{2}\right|=\left|n-m\right|\left|n+m\right|\geq\left|n+m\right|=\text{infinite}$,
it follows that $n^{2}\nsim_{X}m^{2}$. Hence $X$ is thin. On the
other hand, the set $Y=X\cup\left(X+1\right)$ together with the usual
metric is not thin: for any $n^{2},n^{2}+1\in\INF\left(Y\right)$,
their coarse galaxies are $G_{Y}^{c}\left(n^{2}\right)=G_{Y}^{c}\left(n^{2}+1\right)=\set{n^{2},n^{2}+1}$.
\end{example}

\begin{defn}[Standard; \citep{PZ07}]
Let $X$ be a set and $I$ an ideal on (the powerset algebra of)
$X$. The \emph{ideal coarse structure} of $X$ with respect to $I$
is the coarse structure $\mathcal{C}_{I}$ on $X$ generated by the
sets of the form $\Delta_{X}\cup\left(A\times A\right)$, where $A\in I$.
We denote the coarse space $\left(X,\mathcal{C}_{I}\right)$ by $X_{I}$.
\end{defn}

\begin{lem}[{\citep[Remark 3.9]{Ima19}}]
\label{lem:proximity-of-ideal-coarse-space}Let $X$ be a standard
set and $I$ an ideal on $X$. For any $x,y\in\ns{X}$, $x\sim_{X_{I}}y$
if and only if $x=y$ or $x,y\in\bigcup_{A\in I}\ns{A}$.
\end{lem}

\begin{proof}
Suppose $x\sim_{X_{I}}y$. For some $A\in I$, we have $\left(x,y\right)\in\ns{\left(\Delta_{X}\cup\left(A\times A\right)\right)}$.
If $\left(x,y\right)\in\ns{\Delta_{X}}$, then $x=y$. Otherwise,
$\left(x,y\right)\in\ns{\left(A\times A\right)}$, so $x,y\in\ns{A}\subseteq\bigcup_{A\in I}\ns{A}$.

Conversely, suppose $x=y$ or $x,y\in\bigcup_{A\in I}\ns{A}$. If
$x,y\in\bigcup_{A\in I}\ns{A}$, then $x\in\ns{A}$ and $y\in\ns{B}$
for some $A,B\in I$, so $\left(x,y\right)\in\ns{\left(\left(A\cup B\right)\times\left(A\cup B\right)\right)}$.
Hence $\left(x,y\right)\in\ns{\left(\Delta_{X}\cup\left(A\times A\right)\right)}$
for some $A\in I$, and therefore $x\sim_{X_{I}}y$.
\end{proof}
\begin{defn}[Standard; \citep{DPPZ19}]
Let $\left(X,\mathcal{B}_{X}\right)$ be a bornological space. The
bornology $\mathcal{B}_{X}$ is an ideal on $X$. The ideal coarse
structure $\mathcal{C}_{\mathcal{B}_{X}}$ of $X$ with respect to
$\mathcal{B}_{X}$ is called the \emph{satellite coarse structure}
of $X$.
\end{defn}

\begin{lem}[{\citep[Remark 3.9]{Ima19}}]
\label{lem:proximity-of-satellite-coarse-space}Let $X$ be a standard
bornological space. For any $x,y\in\ns{X}$, $x\sim_{X_{\mathcal{B}_{X}}}y$
if and only if $x=y$ or $x,y\in\FIN\left(X\right)$.
\end{lem}

\begin{proof}
This is a special case of \prettyref{lem:proximity-of-ideal-coarse-space}.
\end{proof}
Imagine that the finite part $\FIN\left(X\right)$ is a star and that
the infinite points $\in\INF\left(X\right)$ are satellites around
the star (\prettyref{fig:satellite-coarse-space}). Each satellite
is of infinite distance away from the star and the other satellites.
\begin{figure}
\centering
\begin{tikzpicture}
	\fill[pattern=north west lines] (0, 0) circle[radius=0.5];
	\foreach \n in {0, 1, 2, 3} {
		\fill ({1.6 * cos(\n * 360 / 4)}, {1.6 * sin(\n * 360 / 4)}) circle[radius=1pt];
	}
	\foreach \n in {0, 1, 2, 3, 4, 5, 6, 7, 8} {
		\fill ({2.5 * cos(\n * 360 / 9)}, {2.5 * sin(\n * 360 / 9)}) circle[radius=1pt];
	}
	\draw (2.5, 0) circle (0.45);
	\draw ({2.5 * cos(1 * 360 / 9)}, {2.5 * sin(1 * 360 / 9)}) circle (0.45);
\end{tikzpicture}\caption{\label{fig:satellite-coarse-space}Intuitive picture of the satellite
coarse space}
\end{figure}
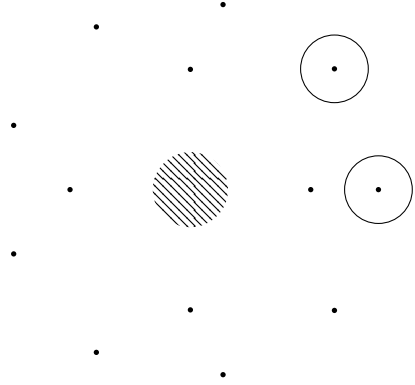

\begin{thm}[{Standard; \citep[Theorem 1]{Pro06}}]
\label{thm:thin-and-satellite}Let $X$ be a connected coarse space.
The following are equivalent:
\begin{enumerate}
\item $X$ is thin;
\item $X=X_{\mathcal{B}_{X}}$.
\end{enumerate}
\end{thm}

\begin{proof}
By \prettyref{thm:nonst-characterisation-of-thin-subsets} and \prettyref{lem:proximity-of-satellite-coarse-space},
$X$ is thin if and only if for any $x,y\in\ns{X}$ we have that
\begin{align*}
x\sim_{X}y & \iff x=y\text{ or }x,y\in\FIN\left(X\right)\\
 & \iff x\sim_{X_{\mathcal{B}_{X}}}y.
\end{align*}
By \citep[Proposition 3.4]{Ima19}, it is also equivalent to ``$X=X_{\mathcal{B}_{X}}$''.
\end{proof}
\begin{proof}[Alternative proof]
It can also be proved by looking at S-boundaries. The S-boundary
$\partial_{S}X_{\mathcal{B}_{X}}$ is a discrete coarse space whose
underlying set is the same as that of $\partial_{S}X$ by \prettyref{lem:proximity-of-satellite-coarse-space}.
We then obtain the following equivalences: $X$ is thin $\iff$ the
S-boundary $\partial_{S}X$ is a discrete coarse space (by \prettyref{cor:thin-and-S-boundary})
$\iff$ $\partial_{S}X=\partial_{S}X_{\mathcal{B}_{X}}$ $\iff$ $X=X_{\mathcal{B}_{X}}$
(by \prettyref{cor:S-corona-determines-coarse-structure}).
\end{proof}
Suppose $X$ is non-thin. There are infinite points whose galaxies
are of cardinality $\geq2$ by \prettyref{thm:nonst-characterisation-of-thin-subsets}.
As we will see below, the galaxy of \emph{some} infinite point can
be divided into two (non-empty) parts by a \emph{standard} set. This
fact is intuitively understandable (see \prettyref{fig:non-satellite-coarse-space}),
but the proof is not obvious and depends on the axiom of choice.
\begin{figure}
\centering
\begin{tikzpicture}
	\fill[pattern=north west lines] (0, 0) circle[radius=0.5];
	\foreach \n in {0, 1, 2, 3, 4, 5, 6, 7, 8} {
		\fill ({2.5 * cos(\n * 360 / 9)}, {2.5 * sin(\n * 360 / 9)}) circle[radius=1pt];
	}
	\foreach \n in {0, 1, 2, 4, 8} {
		\fill ({2.5 * cos(\n * 360 / 9) + 0.5}, {2.5 * sin(\n * 360 / 9) + 0.5}) circle[radius=1pt];
	}
	\draw (2.75, 0.25) circle [x radius=1, y radius=0.5, rotate=45];
	\draw[dashed] (2, 1) -- (3.5, -0.5) node[above right] {$\ns{A}$};
	\
\end{tikzpicture}\caption{\label{fig:non-satellite-coarse-space}Non-satellite coarse space}
\end{figure}
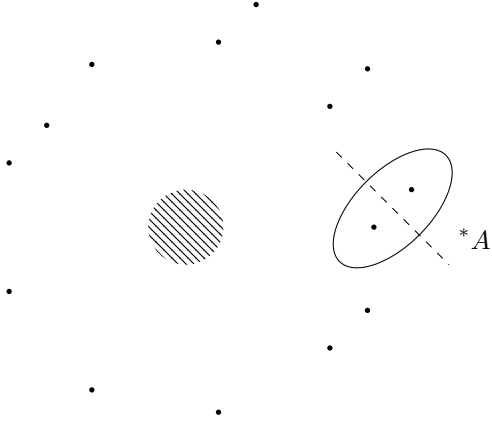

\begin{lem}
\label{lem:dividing-galaxy}Let $X$ be a standard coarse space. If
$\left|G_{X}^{c}\left(x_{0}\right)\right|\geq2$ for some $x_{0}\in\INF\left(X\right)$,
then there exists a subset $A$ of $X$ such that $G_{X}^{c}\left(x\right)\cap\ns{A}\neq\varnothing$
and $G_{X}^{c}\left(x\right)\setminus\ns{A}\neq\varnothing$ for some
$x\in\INF\left(X\right)$.
\end{lem}

\begin{proof}
Fix an $E\in\mathcal{C}_{X}$ with $\left|\ns{E}\left[x_{0}\right]\right|\geq2$
and $\Delta_{X}\subseteq E$. Using Zorn's lemma, take a maximal subset
$Y$ of $X$ such that $\set{E\left[y\right]|y\in Y}$ is disjoint.
Notice that for each $x\in X$ there exists a $y\in Y$ such that
$E\left[x\right]\cap E\left[y\right]\neq\varnothing$ by the maximality.
Set $Y_{0}=\set{y\in Y|\left|E\left[y\right]\right|\geq2}$.

Case I: $\ns{Y_{0}}\cap\INF\left(X\right)\neq\varnothing$. Pick an
$x\in\ns{Y_{0}}\cap\INF\left(X\right)$. By the axiom of choice, we
may choose a subset $A$ of $X$ such that $\left|E\left[y\right]\cap A\right|=1$
for all $y\in Y_{0}$. Then $\left|\ns{E}\left[x\right]\cap\ns{A}\right|=1$
and $\left|\ns{E}\left[x\right]\setminus\ns{A}\right|\geq1$ by transfer.
Hence $G_{X}^{c}\left(x\right)\cap\ns{A}\neq\varnothing$ and $G_{X}^{c}\left(x\right)\setminus\ns{A}\neq\varnothing$.

Case II: $\ns{Y_{0}}\cap\INF\left(X\right)=\varnothing$. Then $x_{0}\notin\ns{Y}$.
(Otherwise, we have that $x_{0}\in\ns{Y_{0}}$ by transfer, a contradiction.)
Define $A=\bigcup_{y\in\ns{Y}}\ns{E}\left[y\right]$. By transfer,
$\ns{E}\left[x_{0}\right]\cap\ns{E}\left[y\right]\neq\varnothing$
for some $y\in\ns{Y}$, i.e., $\ns{E}\left[x_{0}\right]\cap\ns{A}\neq\varnothing$,
so $G_{X}^{c}\left(x_{0}\right)\cap\ns{A}\neq\varnothing$. Suppose,
on the contrary, that $\ns{E}\left[x_{0}\right]\setminus\ns{A}=\varnothing$.
Choose a $y\in\ns{Y}$ so that $x_{0}\in\ns{E}\left[y\right]$. Since
$y\sim_{X}x_{0}\in\INF\left(X\right)$, it follows that $y\notin\ns{Y_{0}}$,
so $\ns{E}\left[y\right]=\set{y}=\set{x_{0}}$ by transfer. This contradicts
with $x_{0}\notin\ns{Y}$. Hence $\ns{E}\left[x_{0}\right]\setminus\ns{A}\neq\varnothing$,
and therefore $G_{X}^{c}\left(x_{0}\right)\setminus\ns{A}\neq\varnothing$.
\end{proof}
Using this lemma, we can easily prove the following two (standard)
characterisations of thinness.
\begin{thm}[{Standard; \citep[Theorem 4]{Pro06}}]
\label{thm:thin-and-SO}Let $X$ be a connected coarse space. The
following are equivalent:
\begin{enumerate}
\item \label{enu:thin-and-SO-function-cond-1}$X$ is thin;
\item \label{enu:thin-and-SO-function-cond-2}every map $f\colon X\to Y$,
where $Y$ is a uniform space, is slowly oscillating;
\item \label{enu:thin-and-SO-function-cond-3}every function $f\colon X\to\set{0,1}$
is slowly oscillating, where $\set{0,1}$ is thought of as a discrete
uniform space.
\end{enumerate}
\end{thm}

\begin{proof}
\eqref{enu:thin-and-SO-function-cond-1}$\Rightarrow$\eqref{enu:thin-and-SO-function-cond-2}:
According to the nonstandard characterisation of slow oscillation
\citep[Theorem 3.30]{Ima19}, it suffices to show that $x\sim_{X}y$
implies $\ns{f}\left(x\right)\approx_{Y}\ns{f}\left(y\right)$ for
all $x,y\in\INF\left(X\right)$. Let $x,y\in\INF\left(X\right)$ and
suppose $x\sim_{X}y$. By \prettyref{thm:nonst-characterisation-of-thin-subsets}
$x=y$, so $\ns{f}\left(x\right)\approx_{Y}\ns{f}\left(y\right)$.

\eqref{enu:thin-and-SO-function-cond-2}$\Rightarrow$\eqref{enu:thin-and-SO-function-cond-3}:
Trivial.

\eqref{enu:thin-and-SO-function-cond-3}$\Rightarrow$\eqref{enu:thin-and-SO-function-cond-1}:
Suppose $X$ is not thin. By \prettyref{thm:nonst-characterisation-of-thin-subsets},
$\left|G_{X}^{c}\left(x_{0}\right)\right|\geq2$ for some $x_{0}\in\INF\left(X\right)$.
By \prettyref{lem:dividing-galaxy}, there exists a subset $A$ of
$X$ such that both $G_{X}^{c}\left(x\right)\cap\ns{A}$ and $G_{X}^{c}\left(x\right)\setminus\ns{A}$
are non-empty for some $x\in\INF\left(X\right)$. Define a function
$f\colon X\to\set{0,1}$ by $f\restriction A\equiv0$ and $f\restriction\left(X\setminus A\right)\equiv1$.
Pick $\xi\in G_{X}^{c}\left(x\right)\cap\ns{A}$ and $\eta\in G_{X}^{c}\left(x\right)\setminus\ns{A}$.
Then $\xi,\eta\in\INF\left(X\right)$, $\xi\sim_{X}\eta$, but $\ns{f}\left(\xi\right)=0\not\approx_{\set{0,1}}1=\ns{f}\left(\eta\right)$
by transfer. By \citep[Theorem 3.30]{Ima19}, $f$ is not slowly oscillating.
\end{proof}
\begin{thm}[{Standard; \citep[Theorem 2.2]{DPPZ19}}]
\label{thm:thin-and-meshy}Let $X$ be a connected coarse space.
The following are equivalent:
\begin{enumerate}
\item $X$ is thin;
\item $\mathcal{M}\left(X\right)\subseteq\mathcal{B}_{X}$.
\end{enumerate}
\end{thm}

\begin{proof}
Suppose $X$ is thin. Let $A$ be an unbounded subset of $X$. Fix
an $x_{0}\in\INF\left(X\right)\cap\ns{A}$ by \citep[Proposition 2.6]{Ima19}.
By \prettyref{thm:nonst-characterisation-of-thin-subsets}, $G_{X}^{c}\left(x_{0}\right)=\set{x_{0}}\subseteq\ns{A}$,
so $x_{0}\in C_{X}^{c}\left(\ns{A}\right)\neq\varnothing$. By \prettyref{thm:nonst-characterisation-of-size-I},
$A$ is not meshy. Hence $\mathcal{M}\left(X\right)\subseteq\mathcal{B}_{X}$.

Conversely, suppose $X$ is not thin. By \prettyref{thm:nonst-characterisation-of-thin-subsets}
and \prettyref{lem:dividing-galaxy}, there exists an $A\subseteq X$
such that $G_{X}^{c}\left(x\right)\cap\ns{A}\neq\varnothing$ and
$G_{X}^{c}\left(x\right)\setminus\ns{A}\neq\varnothing$ for some
$x\in\INF\left(X\right)$. By \citep[Proposition 2.6]{Ima19}, $X=A\cup\left(X\setminus A\right)$
is unbounded (and connected), so either $A$ or $X\setminus A$ is
unbounded. We may assume without loss of generality that $A$ is unbounded.
Fix an $x_{0}\in X$ and define $A'=A\setminus\set{x_{0}}$. Clearly
$A'$ is unbounded too. Let $x\in\ns{X}$. If $x\in\FIN\left(X\right)$,
then $x\sim_{X}x_{0}\notin\ns{A'}$ by \citep[Corollary 3.13]{Ima19},
so $G_{X}^{c}\left(x\right)\nsubseteq\ns{A'}$. If $x\in\INF\left(X\right)$,
then $G_{X}^{c}\left(x\right)\nsubseteq\ns{A}$, so $G_{X}^{c}\left(x\right)\nsubseteq\ns{A'}$.
Hence $C_{X}^{c}\left(\ns{A'}\right)=\varnothing$. By \prettyref{thm:nonst-characterisation-of-size-I},
$A'\in\mathcal{M}\left(X\right)\setminus\mathcal{B}_{X}$.
\end{proof}
Our nonstandard proofs are much simpler (and also intuitive) than
the original standard proofs in \citep{Pro06,DPPZ19}.
\begin{example}
Consider the thin coarse space $X=\set{n^{2}|n\in\mathbb{N}}$. Let
$A$ be any unbounded subset of $X$. Pick an infinite point $n^{2}\in\ns{A}\cap\INF\left(X\right)$,
then $G_{X}^{c}\left(n^{2}\right)=\set{n^{2}}\subseteq\ns{A}$, so
$n^{2}\in C_{X}^{c}\left(\ns{A}\right)\neq\varnothing$. Hence $A$
is not meshy. Next, consider the non-thin coarse space $Y=X\cup\left(X+1\right)$.
Define a function $f\colon Y\to\set{0,1}$ by $f\restriction X\equiv0$
and $f\restriction\left(Y\setminus X\right)\equiv1$, then $\ns{f}\left(n^{2}\right)=0$
and $\ns{f}\left(n^{2}+1\right)=1$ for any $n^{2},n^{2}+1\in\INF\left(Y\right)$,
so $f$ is not slowly oscillating. In this case, $X$ is an unbounded
meshy subset of $Y$, and divides the galaxy $G_{Y}^{c}\left(n^{2}\right)=G_{Y}^{c}\left(n^{2}+1\right)$
into two parts in the sense of \prettyref{lem:dividing-galaxy}.
\end{example}

\subsection{Coarse hyperspaces}

In the rest of this section, we study natural coarse structures on
powersets of coarse spaces, called coarse hyperspaces.

Let $X$ be a metric space. The powerset $\mathcal{P}\left(X\right)$
is endowed with a (generalised) metric $d_{H}\colon X\times X\to\mathbb{R}_{\geq0}\cup\set{+\infty}$,
called the \emph{Hausdorff metric}, defined by
\[
d_{H}\left(A,B\right)=\inf\set{\varepsilon\in\mathbb{R}_{\geq0}|A\subseteq B_{\varepsilon}\text{ and }B\subseteq A_{\varepsilon}},
\]
where $A_{\varepsilon}$ and $B_{\varepsilon}$ are the $\varepsilon$-neighbourhoods
of $A$ and $B$, respectively. The metric space $\left(\mathcal{P}\left(X\right),d_{H}\right)$
is called the \emph{metric hyperspace}. Obviously $\mathcal{P}\left(X\right)$
is equipped with both a uniformity and a coarse structure. This construction
can be generalised to (non-metrisable) uniform spaces and coarse spaces.
\begin{defn}[Standard; \citep{Bou07}]
Let $X$ be a set and $E\subseteq X\times X$. The \emph{exponentiation}
$\exp E$ of $E$ is defined as
\[
\exp E=\set{\left(A,B\right)\in\mathcal{P}\left(X\right)\times\mathcal{P}\left(X\right)|A\subseteq E\left[B\right]\text{ and }B\subseteq E\left[A\right]}.
\]
\end{defn}

The following are evident.
\begin{fact}[{Standard; \citep[Chapter II, p. 34]{Bou07}}]
If $\mathcal{U}_{X}$ is a uniformity on a set $X$, the family $\set{\exp E|E\in\mathcal{U}_{X}}$
generates a uniformity $\exp\mathcal{U}_{X}$ on $\mathcal{P}\left(X\right)$.
\end{fact}

\begin{fact}[{Standard; \citep[Proposition 2.1]{Zav19}}]
If $\mathcal{C}_{X}$ is a coarse structure on a set $X$, the family
$\set{\exp E|E\in\mathcal{C}_{X}}$ generates a coarse structure $\exp\mathcal{C}_{X}$
on $\mathcal{P}\left(X\right)$.
\end{fact}

\begin{defn}[Standard; \citep{DPPZ19,Zav19}]
Let $\left(X,\mathcal{C}_{X}\right)$ be a coarse space and $\mathcal{A}\left(X\right)\subseteq\mathcal{P}\left(X\right)$.
The coarse space $\left(\mathcal{A}\left(X\right),\exp\mathcal{C}_{X}\restriction\mathcal{A}\left(X\right)\right)$
is called the \emph{$\mathcal{A}$-coarse hyperspace}, and is denoted
by $\mathcal{A}\hexp X$. In particular, $\mathcal{P}\hexp X=\left(\mathcal{P}\left(X\right),\exp\mathcal{C}_{X}\right)$
is called the \emph{coarse hyperspace}, and is denoted by $\exp X$.
\end{defn}

First of all, we shall look at the properties of the finite closeness
relation $\sim_{\exp X}$ of $\exp X$.
\begin{lem}
\label{lem:proximity-of-powX}Let $X$ be a standard coarse space.
For any $A,B\in\ns{\left(\mathcal{P}\left(X\right)\right)}$, $A\sim_{\exp X}B$
if and only if $A\subseteq G_{X}^{c}\left(B\right)$ and $B\subseteq G_{X}^{c}\left(A\right)$.
\end{lem}

\begin{proof}
Suppose $A\sim_{\exp X}B$. For some $E\in\mathcal{C}_{X}$, we have
that $A\subseteq\ns{E}\left[B\right]$ and $B\subseteq\ns{E}\left[A\right]$.
So $A\subseteq\ns{E}\left[B\right]\subseteq G_{X}^{c}\left(B\right)$
and $B\subseteq\ns{E}\left[A\right]\subseteq G_{X}^{c}\left(A\right)$.

Conversely, suppose $A\subseteq G_{X}^{c}\left(B\right)$ and $B\subseteq G_{X}^{c}\left(A\right)$.
Then $A\subseteq E\left[B\right]$ and $B\subseteq E\left[A\right]$
hold for all $E\supseteq{\sim_{X}}$. In other words, the internal
subset
\[
\mathcal{E}=\set{E\in\ns{\mathcal{C}_{X}}|A\subseteq E\left[B\right]\text{ and }B\subseteq E\left[A\right]}
\]
of $\ns{\mathcal{C}_{X}}$ contains all (sufficiently small) illimited
elements of $\ns{\mathcal{C}_{X}}$ with respect to $\subseteq$.
By \prettyref{lem:Underspill}, $\mathcal{E}$ has a limited element
$E$, which is bounded by some (standard) $F\in\mathcal{C}_{X}$,
i.e., $E\subseteq\ns{F}$. Hence $\left(A,B\right)\in\mathop{\ns{\exp}}E\subseteq\ns{\left(\exp F\right)}$,
and therefore $A\sim_{\exp X}B$.
\end{proof}
\begin{prop}[{Standard; \citep[Fact 2.3]{Zav19}}]
Let $X$ be a coarse space. The map $\iota_{X}\colon X\to\exp X$
defined by $\iota_{X}\left(x\right)=\set{x}$ is an asymorphic embedding.
\end{prop}

\begin{proof}
Obviously $\iota$ is injective. Let $x,y\in\ns{X}$. Then $x\sim_{X}y$
$\iff$ $\set{x}\subseteq G_{X}^{c}\left(\set{y}\right)$ and $\set{y}\subseteq G_{X}^{c}\left(\set{x}\right)$
$\iff$ $\iota_{X}\left(x\right)\sim_{\exp X}\iota_{X}\left(y\right)$.
By \citep[Theorem 3.23]{Ima19} and \prettyref{thm:nonst-charact-eff-proper},
$\iota$ is effectively proper and bornologous. By \prettyref{prop:effectively-properness-and-inverse},
$\iota$ is an asymorphic embedding.
\end{proof}
\begin{prop}[{Standard; Proof of \citep[Theorem 2.2]{DPPZ19}}]
\label{prop:c-is-bornologous-injection}Let $X$ be an unbounded
coarse space. The map $c_{X}\colon X\to\exp X$ defined by $c_{X}\left(x\right)=X\setminus\set{x}$
is a bornologous injection.
\end{prop}

\begin{proof}
The injectivity is trivial. Let $x,y\in\ns{X}$ and suppose $x\sim_{X}y$
and $x\neq y$. Then $x\in G_{X}^{c}\left(y\right)\subseteq G_{X}^{c}\left(\ns{c_{X}}\left(x\right)\right)$
and $y\in G_{X}^{c}\left(x\right)\subseteq G_{X}^{c}\left(\ns{c_{X}}\left(y\right)\right)$,
so $G_{X}^{c}\left(\ns{c_{X}}\left(x\right)\right)=G_{X}^{c}\left(\ns{c_{X}}\left(y\right)\right)=\ns{X}$.
Hence $\ns{c_{X}}\left(x\right)\sim_{\exp X}\ns{c_{X}}\left(y\right)$.
By \citep[Theorem 3.23]{Ima19}, $c_{X}$ is bornologous.
\end{proof}
\begin{prop}[Standard]
Let $\left(X,\mathcal{U}_{X},\mathcal{C}_{X}\right)$ be an uc-space.
If $\mathcal{U}_{X}$ and $\mathcal{C}_{X}$ are compatible (i.e.
$\mathcal{U}_{X}\cap\mathcal{C}_{X}\neq\varnothing$), then $\exp\mathcal{U}_{X}$
and $\exp\mathcal{C}_{X}$ are compatible.
\end{prop}

\begin{proof}
We denote the uc-hyperspace $\left(\mathcal{P}\left(X\right),\exp\mathcal{U}_{X},\exp\mathcal{C}_{X}\right)$
by $\exp X$. Recall \prettyref{lem:proximity-of-powX}:
\[
A\sim_{\exp X}B\iff A\subseteq G_{X}^{c}\left(B\right)\text{ and }B\subseteq G_{X}^{c}\left(A\right).
\]
Similarly, it is easy to verify the following equivalence:
\[
A\approx_{\exp X}B\iff A\subseteq\mu_{X}^{u}\left(B\right)\text{ and }B\subseteq\mu_{X}^{u}\left(A\right).
\]
Since $\mathcal{U}_{X}$ and $\mathcal{C}_{X}$ are compatible, $\mu_{X}^{u}\left(x\right)\subseteq G_{X}^{c}\left(x\right)$
holds for all $x\in\ns{X}$ by \citep[Theorem 3.30]{Ima19}. Hence
$A\approx_{\exp X}B$ implies $A\sim_{\exp}B$ for all $A,B\in\ns{\left(\mathcal{P}\left(X\right)\right)}$.
By \citep[Theorem 3.30]{Ima19}, $\exp\mathcal{U}_{X}$ and $\exp\mathcal{C}_{X}$
are compatible.
\end{proof}

\subsection{$\flat$-coarse hyperspaces}

Isbell \citep[p. 35]{Isb64} conjectured that if $\mathcal{U}_{1}$
and $\mathcal{U}_{2}$ are distinct (compatible) uniformities on a
topological space $X$, then $\exp\mathcal{U}_{1}$ and $\exp\mathcal{U}_{2}$
induce different topologies on $H\left(X\right)=\set{A\subseteq X|A\colon\text{non-empty closed}}$.
Smith \citep{Smi66} gave a counterexample and some positive results
to this conjecture. On the other hand, the large-scale analogue of
this conjecture is false in any case.
\begin{defn}[Standard; \citep{PP18,Zav19}]
Let $X$ be a coarse space. We denote the family of non-empty bounded
subsets of $X$ by $\flat\left(X\right)$, i.e., $\flat\left(X\right)=\mathcal{B}_{X}\setminus\set{\varnothing}$.
\end{defn}

\begin{lem}
\label{lem:prebornology-of-flat-hyperspace}Let $X$ be a standard
coarse space. For any $A\in\ns{\left(\mathcal{P}\left(X\right)\right)}$
and $B\in\flat\left(X\right)$, the following are equivalent:
\begin{enumerate}
\item $A\sim_{\exp X}\ns{B}$;
\item \label{enu:bexp-cond-2}$\forall a\in A\exists b\in B\left(a\in G_{X}\left(b\right)\right)$
and $\forall b\in B\exists a\in A\left(a\in G_{X}\left(b\right)\right)$.
\end{enumerate}
\end{lem}

\begin{proof}
Suppose $A\sim_{\exp X}\ns{B}$. Let $a\in A$. Since $A\subseteq G_{X}^{c}\left(\ns{B}\right)$,
we can find a $b\in\ns{B}$ so that $a\sim_{X}b$. Since $B$ is non-empty
and bounded, $b\sim_{X}\st{b}$ holds for some $\st{b}\in B$ by \citep[Proposition 2.6]{Ima19}.
Hence $a\in G_{X}\left(\st{b}\right)$ for some $\st{b}\in B$. Next,
let $b\in B$. Since $\ns{B}\subseteq G_{X}^{c}\left(A\right)$, we
can find an $a\in A$ so that $b\sim_{X}a$. Then $a\in G_{X}\left(b\right)$.

Suppose \eqref{enu:bexp-cond-2} holds. By the first half of \eqref{enu:bexp-cond-2},
we have that $A\subseteq\bigcup_{b\in B}G_{X}\left(b\right)=G_{X}^{c}\left(B\right)\subseteq G_{X}^{c}\left(\ns{B}\right)$.
Let $b\in\ns{B}$. Since $B$ is non-empty and bounded, pick a $\st{b}\in B$,
then $b\sim_{X}\st{b}$ by \citep[Proposition 2.6]{Ima19}. By the
last half of \eqref{enu:bexp-cond-2}, we can find an $a\in A$ so
that $a\sim_{X}\st{b}$. Hence $b\in G_{X}^{c}\left(a\right)\subseteq G_{X}^{c}\left(A\right)$.
Therefore $A\sim_{\exp X}\ns{B}$.
\end{proof}
\begin{thm}[Standard]
If $\mathcal{C}_{1}$ and $\mathcal{C}_{2}$ are (compatible) coarse
structures on a prebornological space $X$, then $\exp\mathcal{C}_{1}$
and $\exp\mathcal{C}_{2}$ induce the same prebornology on $\flat\left(X\right)$.
\end{thm}

\begin{proof}
According to \prettyref{lem:prebornology-of-flat-hyperspace}, the
galaxy map $G_{\flat\left(X\right)}$ of $\flat\left(X\right)$ is
determined by the galaxy map $G_{X}$ of $X$. Hence the induced prebornology
of $\flat\left(X\right)$ is determined by the induced prebornology
of $X$ \citep[Propositions 2.6 and 3.12]{Ima19}.
\end{proof}
As a result, it makes sense to consider the \emph{$\flat$-prebornological
hyperspace} $\flat\hexp\left(X\right)$ from a given prebornological
space $X$ (rather than a coarse space) by noting that each prebornological
space admits a compatible coarse structure such as the satellite coarse
structure. If $X$ is a (connected) bornological space, then so is
$\flat\hexp X$.
\begin{prop}[Standard]
\label{prop:flat-hyperspace-is-connected}For every connected coarse
space $X$, the $\flat$-coarse hyperspace $\flat\hexp X$ is connected.
\end{prop}

\begin{proof}
Let $A,B\in\flat\left(X\right)$. Since $X$ is connected, $G_{X}^{c}\left(\ns{A}\right)=G_{X}^{c}\left(\ns{B}\right)=\FIN\left(X\right)$
by \citep[Propositions 2.11 and 3.10]{Ima19}. Hence $\ns{A}\sim_{\flat\hexp X}\ns{B}$.
By \citep[Corollary 3.13]{Ima19}, $\flat\hexp X$ is connected.
\end{proof}
\begin{example}
Recall the coarse spaces $\mathbb{R}$ and $\mathbb{R}'$ in \prettyref{exa:two-coarse-structures-of-R}.
Since $\mathbb{R}$ and $\mathbb{R}'$ have the same bornology, the
$\flat$-coarse hyperspaces $\flat\hexp\mathbb{R}$ and $\flat\hexp\mathbb{R}'$
have the same underlying set. By \prettyref{lem:prebornology-of-flat-hyperspace},
we have that
\[
\FIN\left(\flat\hexp\mathbb{R}\right)=\FIN\left(\flat\hexp\mathbb{R}'\right)=\Set{B\in\ns{\left(\flat\left(\mathbb{R}\right)\right)}|B\subseteq\FIN\left(\mathbb{R}\right)}.
\]
Hence $\flat\hexp\mathbb{R}$ and $\flat\hexp\mathbb{R}'$ have the
same bornology.
\end{example}

\subsection{Coarse hyperspaces and size properties}

Finally, we discuss the relationship between the size properties of
coarse spaces $X$ and their coarse hyperspaces $\mathcal{A}\hexp X$.
\begin{prop}[{Standard; \citep[Remark 2.6]{Zav19}}]
For every coarse space $X$, the $\mathcal{L}$-coarse hyperspace
$\mathcal{L}\hexp X$ is connected, where $\mathcal{L}\left(X\right)$
is the family of all large subsets of $X$.
\end{prop}

\begin{proof}
Let $A$ and $B$ be large subsets of $X$. By \prettyref{thm:nonst-characterisation-of-size-I},
$G_{X}^{c}\left(\ns{A}\right)=G_{X}^{c}\left(\ns{B}\right)=\ns{X}$,
so $\ns{A}\subseteq G_{X}^{c}\left(\ns{B}\right)$ and $\ns{B}\subseteq G_{X}^{c}\left(\ns{A}\right)$,
i.e., $A\sim_{\exp X}B$. By \citep[Corollary 3.13]{Ima19}, $\mathcal{L}\hexp X$
is connected.
\end{proof}
\begin{thm}[{Standard; \citep[Proposition 2.7]{Zav19}}]
For every non-empty connected coarse space $X$, the following are
equivalent:
\begin{enumerate}
\item $X$ is unbounded;
\item $\mathcal{L}\hexp X$ is unbounded.
\end{enumerate}
\end{thm}

\begin{proof}
Suppose $X$ is unbounded. Fix an $x_{0}\in X$. By \prettyref{thm:unboundedness-and-small-subsets},
the singleton $\set{x_{0}}$ is small in $X$, i.e., $X\setminus E\left[x_{0}\right]\in\mathcal{L}\left(X\right)$
holds for all $E\in\mathcal{C}_{X}$. Hence $\ns{X}\setminus E\left[x_{0}\right]\in\ns{\left(\mathcal{L}\left(X\right)\right)}$
holds for all $E\in\ns{\mathcal{C}_{X}}$ by transfer. Now we can
choose an $F\in\ns{\mathcal{C}_{X}}$ so that ${\sim_{X}}\subseteq F$
by \prettyref{lem:General-Approximation-Lemma}. Since $G_{X}^{c}\left(x_{0}\right)\subseteq F\left[x_{0}\right]$,
we have by \prettyref{thm:galaxy-and-core} that
\begin{align*}
G_{X}^{c}\left(\ns{X}\setminus F\left[x_{0}\right]\right) & =\ns{X}\setminus C_{X}^{c}\left(F\left[x_{0}\right]\right)\\
 & \subseteq\ns{X}\setminus C_{X}^{c}\left(G_{X}^{c}\left(x_{0}\right)\right)\\
 & =\ns{X}\setminus G_{X}^{c}\left(x_{0}\right)\\
 & \neq\ns{X}.
\end{align*}
Hence $\ns{X}\setminus F\left[x_{0}\right]\nsim_{\mathcal{L}\hexp X}\ns{X}$.
Therefore $\ns{X}\setminus F\left[x_{0}\right]\in\INF\left(X\right)$.
By \citep[Proposition 2.6]{Ima19}, $\mathcal{L}\hexp X$ is unbounded.

Conversely, suppose $X$ is bounded. Then $G_{X}^{c}\left(A\right)=\ns{X}$
for all non-empty $A\subseteq\ns{X}$ by \citep[Proposition 3.10]{Ima19}.
Hence $A\sim_{\mathcal{L}\hexp X}B$ holds for all $A,B\in\ns{\left(\mathcal{L}\left(X\right)\right)}$
by \prettyref{lem:proximity-of-powX}. (Note that every large subset
of $X$ is non-empty, since $X$ is non-empty.) By \citep[Proposition 3.10]{Ima19},
$\mathcal{L}\hexp X$ is bounded.
\end{proof}
\begin{thm}[{Standard; \citep[Theorem 2.2]{DPPZ19}}]
\label{thm:thin-and-map-c}For every unbounded connected coarse space
$X$, the following are equivalent:
\begin{enumerate}
\item \label{enu:thinness-cond-1}$X$ is thin;
\item \label{enu:thinness-cond-2}$\mathcal{M}'\hexp X$ is connected, where
$\mathcal{M}'\left(X\right)$ is the family of all non-empty meshy
subsets of $X$, i.e., $\mathcal{M}'\left(X\right)=\mathcal{M}\left(X\right)\setminus\set{\varnothing}$;
\item \label{enu:thinness-cond-3}the map $c_{X}\colon X\to\exp X$ defined
by $c_{X}\left(x\right)=X\setminus\set{x}$ is an asymorphic embedding.
\end{enumerate}
\end{thm}

\begin{proof}
\eqref{enu:thinness-cond-1}$\Rightarrow$\eqref{enu:thinness-cond-2}:
This part is purely standard. By \prettyref{thm:thin-and-meshy},
$\mathcal{M}\left(X\right)\subseteq\mathcal{B}_{X}$, so $\mathcal{M}'\left(X\right)\subseteq\flat\left(X\right)$.
Since $\flat\hexp X$ is connected by \prettyref{prop:flat-hyperspace-is-connected},
the subspace $\mathcal{M}'\hexp X$ is connected too.

\eqref{enu:thinness-cond-2}$\Rightarrow$\eqref{enu:thinness-cond-1}:
Let $A\in\mathcal{M}'\left(X\right)$. Fix an $x_{0}\in X$. Since
$\set{x_{0}}\in\flat\left(X\right)$, we have $\set{x_{0}}\in\mathcal{M}'\left(X\right)$
by \prettyref{cor:bounded-implies-meshy}. $\mathcal{M}'\hexp X$
is connected, so $\ns{A}\sim_{\exp X}\set{x_{0}}$ by \citep[Corollary 3.13]{Ima19},
i.e., $\ns{A}\subseteq G_{X}^{c}\left(x_{0}\right)$ (and $x_{0}\in G_{X}^{c}\left(\ns{A}\right)$).
By \citep[Proposition 2.6]{Ima19}, $A\in\flat\left(X\right)$. Hence
$\mathcal{M}\left(X\right)\subseteq\mathcal{B}_{X}$.

\eqref{enu:thinness-cond-1}$\Rightarrow$\eqref{enu:thinness-cond-3}:
According to \prettyref{prop:c-is-bornologous-injection}, it suffices
to show that $c_{X}$ is effectively proper. Let $x,y\in\ns{X}$ and
suppose $x\nsim_{X}y$. Either $x\in\INF\left(X\right)$ or $y\in\INF\left(X\right)$
holds by \citep[Corollary 3.13]{Ima19}. We may assume without loss
of generality that $x\in\INF\left(X\right)$. By \prettyref{thm:nonst-characterisation-of-thin-subsets},
$x\notin G_{X}^{c}\left(z\right)$ for any $z\in\ns{X}\setminus\set{x}$,
so $x\in\ns{c_{X}}\left(y\right)\nsubseteq G_{X}^{c}\left(\ns{c_{X}}\left(x\right)\right)$,
and therefore $\ns{c_{X}}\left(x\right)\nsim_{\exp X}\ns{c_{X}}\left(y\right)$.
By \prettyref{thm:nonst-charact-eff-proper}, $c_{X}$ is effectively
proper.

\eqref{enu:thinness-cond-3}$\Rightarrow$\eqref{enu:thinness-cond-1}:
Suppose $X$ is not thin. Fix an $x\in\FIN\left(X\right)$. By \prettyref{thm:nonst-characterisation-of-thin-subsets},
there exists a $y\in\INF\left(X\right)$ such that $\left|G_{X}^{c}\left(y\right)\right|\geq2$.
It is easy to see that $G_{X}^{c}\left(\ns{c_{X}}\left(x\right)\right)=\ns{X}$
and $G_{X}^{c}\left(\ns{c_{X}}\left(y\right)\right)=\ns{X}$, so $\ns{c_{X}}\left(x\right)\sim_{\exp X}\ns{c_{X}}\left(y\right)$.
However, since $x\in\FIN\left(X\right)$ and $y\in\INF\left(X\right)$,
we have $x\nsim_{X}y$. By \prettyref{thm:nonst-charact-eff-proper},
$c_{X}$ is not effectively proper.
\end{proof}
\appendix

\section{\label{sec:Overspill-and-underspill}Overspill and underspill principles
for directed sets}

\renewcommand*{\appendixname}{}
\begin{defn}
Let $\varDelta$ be a standard directed set. An element $\delta\in\ns{\varDelta}$
is said to be \emph{limited} if $\delta$ is bounded by some element
of $\varDelta$ (i.e. $\delta\leq\gamma$ for some $\gamma\in\varDelta$);
and $\delta$ is \emph{illimited} if $\delta$ bounds $\varDelta$
(i.e. $\gamma\leq\delta$ for all $\gamma\in\varDelta$).
\end{defn}

Note that limitedness and illimitedness are not the negations of each
other. If $\varDelta$ is not linearly ordered, $\ns{\varDelta}$
may have elements which are neither limited nor illimited. If $\varDelta$
is self-bounded, $\ns{\varDelta}$ has elements which are both limited
and illimited.
\begin{lem}
\label{lem:General-Approximation-Lemma}Let $\varDelta$ be a standard
directed set. Then $\ns{\varDelta}$ has an illimited element.
\end{lem}

\begin{proof}
Since $\varDelta$ is directed, for each finite subset $A$ of $\varDelta$,
there exists a $\delta\in\varDelta$ such that $\gamma\leq\delta$
for all $\gamma\in A$. By weak saturation, there exists a $\delta\in\ns{\varDelta}$
such that $\gamma\leq\delta$ for all $\gamma\in\varDelta$.
\end{proof}
\begin{example}
\begin{enumerate}
\item The illimited elements of $\ns{\left(\mathbb{R},\leq\right)}$ are
precisely the positive infinite hyperreals.
\item The illimited elements of $\ns{\left(\mathbb{R}_{+},\geq\right)}$
are precisely the positive infinitesimal hyperreals.
\item Let $\left(X,\mathcal{B}_{X}\right)$ be a standard prebornological
space and $x\in X$. The family $\mathcal{BN}_{X}\left(x\right)=\set{B\in\mathcal{B}_{X}|x\in B}$
is directed with respect to $\subseteq$. The illimited elements of
$\ns{\mathcal{BN}_{X}\left(x\right)}$ are precisely the elements
of $\ns{\mathcal{B}_{X}}$ containing the galaxy $G_{X}\left(x\right)$.
See also \citep[Lemma 2.5]{Ima19}.
\item Let $\left(X,\mathcal{C}_{X}\right)$ be a standard coarse space.
$\mathcal{C}_{X}$ is directed with respect to $\subseteq$. The illimited
elements of $\ns{\mathcal{C}_{X}}$ are precisely the elements of
$\ns{\mathcal{C}_{X}}$ containing the finite closeness relation $\sim_{X}$.
See also \citep[Lemma 3.3]{Ima19}.
\end{enumerate}
\end{example}

\begin{lem}[Overspill Principle]
\label{lem:Overspill}Let $\varDelta$ be a standard directed set
and $A$ an internal subset of $\ns{\varDelta}$.
\begin{enumerate}
\item If $A$ contains all sufficiently large limited elements of $\ns{\varDelta}$,
then it also contains all sufficiently small illimited elements of
$\ns{\varDelta}$.
\item If $A$ contains arbitrarily large limited elements of $\ns{\varDelta}$,
then it also contains arbitrarily small illimited elements of $\ns{\varDelta}$.
\end{enumerate}
\end{lem}

\begin{proof}
\begin{enumerate}
\item Suppose that $A$ contains all limited elements of $\ns{\varDelta}$
above $L_{0}\in\varDelta$. For $L\in\varDelta$, set $A_{L}:=\set{U\in\ns{\varDelta}|L\leq U\wedge\left[L,U\right]\subseteq A}$.
By assumption, the family $\set{A_{L}|L\in\varDelta,L\geq L_{0}}$
has the finite intersection property. Hence we can pick an element
$U\in\bigcap_{L\in\varDelta,L\geq L_{0}}A_{L}$ by saturation. The
element $U$ is illimited in $\ns{\varDelta}$ and every illimited
element of $\ns{\varDelta}$ below $U$ belongs to $A$.
\item Let $U\in\ns{\varDelta}$ be illimited. For $L\in\varDelta$ set $B_{L}:=\left[L,U\right]\cap A$.
By assumption, the family $\set{B_{L}|L\in\varDelta}$ has the finite
intersection property. Hence we can pick an element $\delta\in\bigcap_{L\in\varDelta}B_{L}$
by saturation. $\delta$ is an illimited element of $\ns{\varDelta}$
below $U$ and belongs to $A$.\qedhere
\end{enumerate}
\end{proof}
\begin{lem}[Underspill Principle]
\label{lem:Underspill}Let $\varDelta$ be a standard directed set
and $A$ an internal subset of $\ns{\varDelta}$.
\begin{enumerate}
\item If $A$ contains all sufficiently small illimited elements of $\ns{\varDelta}$,
then it also contains all sufficiently large limited elements of $\ns{\varDelta}$.
\item If $A$ contains arbitrarily small illimited elements of $\ns{\varDelta}$,
then it also contains arbitrarily large limited elements of $\ns{\varDelta}$.
\end{enumerate}
\end{lem}

\begin{proof}
Apply the contraposition of \prettyref{lem:Overspill} to the complement
$\ns{\varDelta}\setminus A$.
\end{proof}
\begin{rem}
The overspill principle can be generalised to \textbf{boldface monadic}
subsets, while the underspill principle can be generalised to \textbf{boldface
galactic} subsets (see e.g. \citep{ImaXX}). A set $M$ is said to
be \textbf{\emph{boldface monadic}} if there is a family $\mathcal{M}=\set{M_{i}|i\in S}$
of internal sets, where $S$ is standard, such that $M=\bigcap\mathcal{M}$;
and a set $G$ is said to be \textbf{\emph{boldface galactic}} if
there is a family $\mathcal{G}=\set{G_{i}|i\in S}$ of internal sets,
where $S$ is standard, such that $G=\bigcup\mathcal{G}$. Here $\mathcal{G}$
and $\mathcal{H}$ themselves are not necessarily internal. These
boldface properties are different from the lightface properties considered
in \citep[Definition 2.7 and Remark 3.7]{Ima19}: $M$ is said to
be \emph{lightface monadic} if there is a family $\mathcal{M}=\set{M_{i}|i\in S}$
of standard sets such that $M=\bigcap_{i\in S}\ns{M_{i}}$; and $G$
is said to be \emph{lightface galactic} if there is a family $\mathcal{G}=\set{G_{i}|i\in S}$
of standard sets such that $G=\bigcup_{i\in S}\ns{G_{i}}$.
\end{rem}

\bibliographystyle{IEEEtranSN}
\bibliography{NMLT_II}

\end{document}